\documentclass[12pt,twoside,pra,aps,amssymb,amsfonts,amsmath,tightenlines]{revtex4}

\usepackage{amssymb}
\usepackage{amsmath}
\usepackage{amsfonts}
\usepackage{amsbsy}
\usepackage{amstext}
\usepackage{dsfont}
\usepackage{graphicx}

\newtheorem{theorem}{Theorem}

\newtheorem{proposition}[theorem]{Proposition}

\newtheorem{definition}{Definition}

\newenvironment{proof}[1][Proof.]{\begin{trivlist}
\item[\hskip \labelsep {\bfseries #1}]}{\end{trivlist}}

\newcommand{\qed}{\hfill \ensuremath{\Box}}

\newcommand{\half}{\mbox{$\textstyle \frac{1}{2}$}}

\newcommand{\re}{\mbox{$\rm e$}}

\newcommand{\rd}{\mbox{$\rm d$}}

\begin{document}

\title{Signal processing with L\'evy information}

\author {\textsc{Dorje C.~Brody$^{1}$, Lane P.~Hughston$^{2}$, Xun Yang$^{3,4}$}}
\affiliation{$^{1}$Mathematical Sciences, Brunel University, Uxbridge UB8 3PH, UK \\ 
$^{2}$Department of Mathematics, University College London, London WC1E Ê6BT, UK \\ 
$^{3}$Department of Mathematics, Imperial College London, London SW7 2BZ, UK \\ 
$^{4}$Shell International Limited, Shell Centre, London SE1 7NA, UK}

\begin{abstract}
L\'evy processes, which have stationary independent increments, are ideal for 
modelling the various types of noise that can arise in communication channels.  If a L\'evy 
process admits exponential moments, then there exists a parametric family of measure 
changes called Esscher transformations. If the parameter is replaced with an independent 
random variable, the true value of which represents a ``message'', then under the transformed 
measure the original L\'evy process takes on the character of an ``information process''. In this 
paper we develop a theory of such L\'evy information processes. The underlying L\'evy process, 
which we call the fiducial process, represents the ``noise type''. Each such noise type is capable 
of carrying a message of a certain specification. A number of examples are worked out in detail, 
including information processes of the Brownian, Poisson, gamma, variance gamma, negative 
binomial, inverse Gaussian, and normal inverse Gaussian type. Although in general there is no 
additive decomposition of information into signal and noise, one is led nevertheless for each 
noise type to a well-defined scheme for signal detection and enhancement relevant to a variety 
of practical situations.
\\ 



\end{abstract}

\maketitle

\section{Introduction}

\noindent The idea of filtering the noise out of a noisy message as a way of increasing 
its information content is illustrated by Norbert Wiener in his book {\em Cybernetics} 
(Wiener 1948)  by means of the following example. The true message is represented by a 
variable $X$ which has a known probability distribution. An agent wishes to determine as best 
as possible the value of $X$, but due to the presence of noise the agent can only observe 
a noisy version of the message of the form $\xi=X+\epsilon$, where $\epsilon$ is
independent of $X$. Wiener shows how, given the observed value of 
the noisy message $\xi$, the original distribution of $X$ can be transformed into an improved 
\textit{a posteriori} distribution that has a higher information content. 
The \textit{a posteriori} distribution can then be used to determine a best estimate for the value 
of $X$. 

The  theory of filtering was developed in the 1940s when the inefficiency of anti-aircraft fire 
made it imperative to introduce effective filtering-based devices (Wiener 1949, 1954). A 
breakthrough came with the work of Kalman, who reformulated the theory in a manner more 
well-suited for dynamical state-estimation problems (Kailath 1974, Davis 1977). This period 
coincided with the emergence of the modern control theory of Bellman and Pontryagin 
(Bellman 1961, Pontryagin \textit{et al}. 1962). Owing to the importance of its applications, 
much work has been carried out since then. According to an estimate of Kalman (1994), 
over 200,000 articles and monographs had been published on applications of the Kalman 
filter alone. The theory of stochastic filtering, in its modern form, is not much different 
conceptually from the elementary example described by Wiener in the 1940s. The message, 
instead of being represented by a single variable, in the general setup can take the form of a 
time series (the ``signal'' or ``message'' process). The information made available to the 
agent also takes the form of a time series (the ``observation'' or ``information'' process), 
typically given by the sum of two terms, the first being a functional of the signal process, 
and the second being a noise process. The nature of the signal process can be rather 
general, but in most applications the noise is chosen to be a Wiener process  (see, e.g., 
Liptser \& Shiryaev 2000, Xiong 2008, Bain \& Crisan 2010). There is no reason 
\textit{a priori}, however, why an information process should be ``additive'', 
or even why it should be given as a functional of a signal process and a 
noise process. From a mathematical perspective, it seems that the often proposed ansatz 
of an additive decomposition of the observation process is well-adapted to the situation 
where the noise is Gaussian, but  is not so natural when the noise is
discontinuous. Thus while a good deal of recent research has been carried out on 
the problem of  filtering noisy information containing jumps (see, e.g., Rutkowski 1994, Ahn 
\& Feldman 1999, Meyer-Brandis \& Proske 2004, Poklukar 2006, Popa \& Sritharan 2009, 
Grigelionis \& Mikulevicius 2011, and references cited therein), such work has usually been 
pursued under the assumption of an additive relation between signal and noise, and it is not 
unreasonable to ask whether a more systematic treatment of the problem might be available 
that involves no presumption of additivity and that is more naturally adapted to the 
mathematics of the situation. 

The purpose of the present paper is to introduce a broad  class of information processes suitable 
for modelling situations involving discontinuous signals, discontinuous noise, and discontinuous 
information. No assumption is made to the effect that information can be expressed as a function 
of signal and noise.  Instead, information processes are classified according to their  ``noise type''.  
Information processes of the same noise type are then distinguished from one another by the 
messages that they carry. Each noise type is associated to a 
L\'evy process, which we call the \textit{fiducial process}. The fiducial process is the information 
process that results for a given noise type in the case of a null message, and can be thought of as 
a ``pure noise'' process of that noise type. Information processes can then be classified by the 
characteristics of the associated fiducial processes. To keep the discussion elementary, 
we consider the case of a one-dimension fiducial process and examine the situation where the 
message is represented by a single random variable. The goal is to construct the optimal filter for 
the class of information processes that we consider in the form of a map that takes the  \textit{a 
priori} distribution of the message to an  \textit{a posteriori} distribution that depends on the 
information that has been made available. A number of examples will be presented. The results 
vary remarkably in detail and character for the different types of filters considered, and yet there 
is an overriding unity in the general scheme, which allows for the construction of a multitude of 
examples and applications. 

A synopsis of the main ideas, which we develop more fully in the remainder of the paper, can 
be presented as follows. We recall the idea of the Esscher transform as a change of 
probability measure on a probability space $({\mathit\Omega},{\mathcal F}, {\mathbb P_0})$ 
that supports a  L\'evy process $\{\xi_t\}_{t \geq 0}$ that possesses ${\mathbb P}_0$-exponential 
moments. The space of admissible moments is  the set 
$A = \{w \in \mathds R :  {\mathbb E}^{\mathbb P_0}[\exp(w \xi_t)] < \infty \}$. The associated 
L\'evy exponent $\psi(\alpha) = t^{-1} \ln{\mathbb E}^{{\mathbb P}_0}[\exp(\alpha \xi_t)]$ then 
exists for all $\alpha \in A_{\mathds C} :=\{ w \in {\mathds C} \,:\, {\rm Re}\, w \in A\}$, and does not
depend on $t$.  A parametric family of measure changes ${\mathbb P}_0 
\rightarrow {\mathbb P}_{\lambda}$ commonly called Esscher transformations can be 
constructed by use of the exponential martingale family 
$\{\rho^{\lambda}_t\}_{t \geq 0}$, defined for each  $\lambda \in A$ by $\rho_t^\lambda = \exp 
\left({\lambda \xi_t - \psi(\lambda) t}\right)$. If $\{\xi_t\}$ is a ${\mathbb P}_0$-Brownian motion, 
then $\{\xi_t\}$ is $\mathbb P_{\lambda}$-Brownian with drift $\lambda$; if $\{\xi_t\}$ 
is a $\mathbb P_0$-Poisson process with intensity $m$, then $\{\xi_t\}$ is 
${\mathbb P}_{\lambda}$-Poisson with intensity $\re^{\lambda} m$; if $\{\xi_t\}$ is a 
$\mathbb P_0$-gamma process with rate parameter $m$ and scale parameter $\kappa$, 
then $\{\xi_t\}$ is ${\mathbb P}_{\lambda}$-gamma with rate parameter $m$ and scale 
parameter $\kappa / (1 - \lambda)$. Each case is different in character. 
A natural generalisation of the Esscher transform results when the parameter 
$\lambda$ in the measure change is replaced by a random variable $X$. From the 
perspective of the new measure ${\mathbb P}_X$, the process $\{\xi_t\}$ retains the 
``noisy'' character of its  $\mathbb P_0$-L\'evy origin, but also carries information about $X$. 
In particular, if one assumes that $X$ and $\{\xi_t\}$ are ${\mathbb P}_0$-independent, and 
that the support of $X$ lies in $A$, then we say that $\{\xi_t\}$ defines a {\em L\'evy 
information process} under ${\mathbb P}_X$ carrying the message $X$. Thus, the 
change of measure inextricably intertwines signal and noise. More 
abstractly, we say that on a probability space $({\mathit\Omega},{\mathcal F}, {\mathbb P})$ 
a random process $\{\xi_t\}$ is a L\'evy information process with message (or ``signal'') $X$ 
and noise type (or ``fiducial exponent'') $\psi_0(\alpha)$ if $\{\xi_t\}$ is {\em conditionally} a 
$\mathbb P$-L\'evy given  $X$, with L\'evy exponent $\psi_0(\alpha + X) - 
\psi_0(X)$ for $\alpha \in \mathds C^{\rm I} := \{w \in \mathds C : {\rm Re} \, w = 0\} $. We 
are thus able to classify L\'evy information processes by their noise type, and for 
each noise type we can specify the class of random variables that are admissible as signals 
that can be carried in the environment of such noise. We consider a number of different 
noise types, and construct explicit representations of the associated information processes. 
We also derive an expression for the optimal filter in the general situation, which 
transforms the {\em a priori} distribution of the signal to the improved {\em a posteriori} 
distribution that can be inferred on the basis of received information. 

The plan of the paper is as follows. In Section~\ref{sec:2}, after recalling some facts about 
processes with stationary and independent increments, we define L\'evy information, and in 
Proposition~\ref{prop:1} we show that  the signal carried by a L\'evy information process is 
effectively ``revealed'' after the passage of sufficient time.  In Section~\ref{sec:3} we present 
in Proposition~\ref{prop:2} an explicit  construction using a change of measure technique that 
ensures the existence of L\'evy information processes, and in Proposition~\ref{prop:3} we 
prove a converse to the effect that any L\'evy information process can be obtained in this way. 
In Proposition~\ref{prop:4} we construct the optimal filter for general L\'evy information 
processes, and in Proposition~\ref{prop:5} we show that such processes have the Markov 
property. In Proposition~\ref{prop:6} we establish a result that indicates in more detail how the information 
content of the signal is coded into the structure of an information process. Then in 
Proposition~\ref{prop:7} we present a general construction of the so-called innovations process 
associated with L\'evy information. Finally in Section~\ref{sec:4} we proceed to examine 
a number of specific examples of L\'evy information processes, for which explicit 
representations are constructed in Propositions~\ref{prop:8}--\ref{prop:15}. 

\section{L\'evy information}
\label{sec:2}

\noindent We assume that the reader is familiar with the theory of L\'evy processes 
(Bingham 1975, Sato 1999, Appelbaum 2004, Bertoin 2004, Protter 2005, Kyprianou 2006). 
For an overview of some of the specific L\'evy processes considered later in this paper we 
refer the reader to Schoutens (2003). A real-valued process $\{\xi_t\}_{t \geq 0}$ on a 
probability space $({\mathit\Omega},{\mathcal F},{\mathbb P})$ is a L\'evy process if: (i) 
$\mathbb P (\xi_0 =0)=1$, (ii) $\{\xi_t\}$ has stationary and independent increments, (iii) 
$\lim_{t \to s} \mathbb P (|\xi_t-\xi_s| >\epsilon) = 0$, and (iv) $\{\xi_t\}$ is almost surely 
c\`adl\`ag. For a L\'evy process $\{\xi_t\}$ to give rise to a class of 
information processes, we require that it should possess exponential moments. 
Let us consider the set defined for some (equivalently for all) $t >0$ by 
\begin{eqnarray}
A= \left\{ w \in {\mathds R} \,:\, {\mathbb E}^{\mathbb P}[\exp(w \xi_t)] < \infty \right\}.
\end{eqnarray}
If $A$ contains points other than $w = 0$, then we say that $\{\xi_t\}$ possesses exponential 
moments. We define a function $\psi : A \to {\mathds R}$ called the L\'evy exponent (or 
cumulant function), such that
\begin{eqnarray}
{\mathbb E}^{\mathbb P}\left[\exp(\alpha \,\xi_t)\right] = \exp(\psi(\alpha) \,t)
\label{exponent}
\end{eqnarray}
for $\alpha \in A$. If a L\'evy process possesses exponential moments, then an exercise
shows that $\psi(\alpha)$ is convex on $A$, that the mean and variance of $\xi_t$ are given 
respectively by $\psi'(0)\,t$ and $\psi''(0)\,t$, and that as a consequence of the convexity of 
$\psi(\alpha)$ the marginal exponent $\psi'(\alpha)$ possesses a unique inverse $I(y)$ such 
that $I(\psi'(\alpha))=\alpha$ for $\alpha\in A$. The L\'evy exponent extends to a function 
$\psi : A_{\mathds C} \to {\mathds C}$  where 
$A_{\mathds C} =\{ w \in {\mathds C} \,:\, {\rm Re}\,w \in A\}$,
and it can be shown (Sato 1999, Theorem 25.17) that $\psi(\alpha)$ admits a 
L\'evy-Khintchine representation of the form
\begin{eqnarray}
\psi(\alpha) = p\alpha + \frac{1}{2}q \alpha^2 + \int_{{\mathds R} \backslash \{0\}} 
(\re^{\alpha z} -1 -\alpha z {\mathds 1}\{|z|<1\}) \nu(\rd z)
\label{eq:2}
\end{eqnarray}
with the property that (\ref{exponent}) holds for for all $\alpha \in A_{\mathds C}$. Here 
${\mathds 1}\{\cdot\}$ denotes the indicator function, $p\in{\mathds R}$ and $q\geq 0$ 
are constants, and the so-called L\'evy measure $\nu(\rd z)$ is a positive measure defined 
on ${\mathds R} \backslash \{0\}$ satisfying 
\begin{eqnarray}
\int_{{\mathds R} \backslash \{0\}} (1 \wedge z^2 )\nu(\rd z) < \infty.
\end{eqnarray}
If the L\'evy process possesses exponential moments, then for $\alpha \in A$ we also have
\begin{eqnarray}
\int_{{\mathds R} \backslash \{0\}} \re^{\alpha z} \, {\mathds 1}\{|z| \geq 1\} \, \nu(\rd z) < \infty .
\end{eqnarray}
The L\'evy measure has the following interpretation: if $B$ is a measurable subset of 
${\mathds R}\backslash \{0\} $, then $\nu(B)$ is the rate at which jumps arrive for which the 
jump size lies in $B$. Consider  the sets defined for $n \in {\mathds N}$ by $B_n = \{ z \in 
{\mathds R} \, | \, 1/n \leq |z| \leq 1 \}$. If $\nu(B_n)$ tends to infinity for large $n$ we say 
that $\{\xi_t\}$ is a process of  infinite activity, meaning that the rate of arrival of small jumps 
is unbounded. If $\nu({\mathds R} \backslash \{0\}) < \infty$ one says that $\{\xi_t\}$ has 
finite activity. We refer to the data $K = (p, q, \nu)$ as the characteristic triplet (or
``characteristic'') of the associated L\'evy process. Thus we can classify a L\'evy process 
abstractly by its characteristic $K$, or, equivalently,  its exponent 
$\psi(\alpha)$. This means one can speak of a ``type'' of L\'evy noise by reference to the 
associated characteristic or exponent. 

Now suppose we fix a measure $\mathbb P_0$ on a measurable space 
$({\mathit\Omega},{\mathcal F})$, and let $\{\xi_t\}$ be $\mathbb P_0$-L\'evy, with 
exponent $\psi_0(\alpha)$. There exists a parametric family of probability measures 
$\{\mathbb P_{\lambda}\}_{\lambda\in A}$ on $({\mathit\Omega},{\mathcal F})$ such that for 
each choice of $\lambda$ the process $\{\xi_t\}$ is ${\mathbb P}_{\lambda}$-L\'evy. The 
changes of measure arising in this way are called Esscher transformations (Esscher 1932, 
Gerber \& Shiu 1994, Chan 1999, Kallsen \& Shiryaev 2002, Hubalek \& Sgarra 2006). Under 
an Esscher transformation the characteristics of a L\'evy process are transformed from one 
type to another, and one can speak of a ``family'' of L\'evy processes interrelated by Esscher 
transformations. The relevant change of measure can be specified by use of the process 
$\{\rho_t^\lambda\}$ defined for $\lambda \in A$ by 
\begin{eqnarray}
\rho_t^\lambda := \left. \frac{\rd{\mathbb P}_\lambda}
{\rd{\mathbb P}_0}\right|_{{\mathcal F}_t} = 
\exp\left({\lambda \xi_t - \psi_0(\lambda) t}\right) , 
\label{Esscher martingale}
\end{eqnarray}
where ${\mathcal F}_t  = \sigma \left[\{\xi_s\}_{0 \leq s \leq t}\,\right]$. One can check that 
$\{\rho_t^\lambda\}$ is an $( \{ {\mathcal F}_t \} , \mathbb P_0)$-martingale: indeed, as a 
consequence of the fact that $\{\xi_t\}$ has stationary and independent increments we have
\begin{eqnarray}
{\mathbb E}^{\mathbb P_0}_s[\rho_t^\lambda] = 
{\mathbb E}^{\mathbb P_0}_s[\re^{\lambda (\xi_t-\xi_s)}] \, \re^{\lambda \xi_s-t 
\psi_0(\lambda)} 
= \rho_s^\lambda 
\end{eqnarray}
for $s \leq t$, where ${\mathbb E}^{\mathbb P_0}_t[\,\cdot\,]$ denotes conditional 
expectation under $\mathbb P_0$ with respect to $\mathcal F_t $. It is straightforward to 
show that $\{\xi_t\}$ has ${\mathbb P}_{\lambda}$-stationary and independent increments, 
and that the ${\mathbb P}_{\lambda}$-exponent of $\{\xi_t\}$, which is  defined on the set 
$A_{\mathds C}^{\lambda} :=\{w \in {\mathds C}\, | \, {\rm Re} \, w + \lambda \in A\}$, is 
given by 
\begin{eqnarray}
\psi_\lambda(\alpha) := t^{-1} \ln  {\mathbb E}^{{\mathbb P}_{\lambda}}
[\exp(\alpha \xi_t)] = \psi_0(\alpha + \lambda) - \psi_0(\lambda),
\end{eqnarray}
from which by use of the L\'evy-Khintchine representation (\ref{eq:2}) one can work out the 
characteristic triplet $K_\lambda$ of $\{\xi_t\}$ under ${\mathbb P}_{\lambda}$. We observe 
that if the Esscher martingale (\ref{Esscher martingale}) is expanded as a power series in 
$\lambda$, then the resulting coefficients, which are given by polynomials in $\xi_t$ and $t$, 
form a so-called Sheffer set (Schoutens \& Teugels 1998), each element of which defines 
an $( \{ {\mathcal F}_t \} , \mathbb P_0)$-martingale. The first three of these polynomials take 
the form $Q^1(x,t)= x - \psi't$, $Q^2(x,t)=\frac{1}{2} [(x - \psi't)^2 - \psi''t]$, and
$Q^3(x,t)=\frac{1}{6} [(x - \psi't)^3 - 3\psi''t (x - \psi't)- \psi'''t]$, where  $\psi' = \psi_0'(0)$, 
$\psi'' = \psi_0''(0)$, and $\psi''' = \psi_0'''(0)$.
The corresponding polynomial L\'evy-Sheffer  martingales are 
given by $Q^1_t = Q^1(\xi_t,t)$, $Q^2_t = Q^2(\xi_t,t)$, and $Q^3_t = Q^3(\xi_t,t)$.

In what follows we use the terms ``signal'' and ``message'' interchangeably. We write 
$\mathds C^{\rm I} = \{w \in \mathds C : {\rm Re} \, w = 0\} $. For any random variable $Z$ 
on $({\mathit\Omega},{\mathcal F}, {\mathbb P})$ we write ${\mathcal F}^Z = \sigma[Z]$, 
and when it is convenient we write ${\mathbb E}^{{\mathbb P}}[ \,\cdot \,| Z]$ for 
${\mathbb E}^{\mathbb P}[ \,\cdot \,| {\mathcal F}^Z]$. For processes we use both of the  
notations $\{Z_t\}$ and $\{Z(t)\}$, depending on the context.

With these background remarks in mind, we are in a position to define a \textit{L\'evy 
information process}. We confine the discussion to the case of a ``simple'' message, 
represented by a random variable $X$. In the situation when the noise is Brownian 
motion, the information admits a linear decomposition into signal and noise. In the general 
situation the relation between signal and noise is more subtle, and has the 
character of a fibre space, where one thinks of the points of the base 
space as representing the different noise types, and the points of the fibres as corresponding 
to the different information processes that one can construct in association with a given 
noise type. Alternatively, one can think of the base as being the convex space of L\'evy characteristics, and the fibre over a given point of the base as the convex space
of messages that are compatible with the associated noise type.

We fix a probability space $({\mathit\Omega},{\mathcal F}, {\mathbb P})$, and an Esscher 
family of L\'evy characteristics $K_{\lambda}$, $\lambda \in A$, with associated L\'evy 
exponents $\psi_\lambda(\alpha)$, $\alpha \in A_{\mathds C}^{\lambda} $. We refer to $K_0$ 
as the fiducial characteristic, and $\psi_0(\alpha)$ as the fiducial exponent. The intuition here 
is that the abstract  L\'evy process of characteristic $K_0$ and exponent $\psi_0(\alpha)$, which 
we call the ``fiducial'' process, represents the noise type of the associated  information process. 
Thus we can use $K_0$, or equivalently $\psi_0(\alpha)$, to label the noise type. 
\begin{definition}
By a L\'evy information process with fiducial characteristic $K_0$, carrying the  message $X$, 
we mean a random process $\{\xi_t\}$, together with a random variable $X$, such that $\{\xi_t\}$ 
is conditionally $K_X$-L\'evy given $\mathcal F^X$.
\label{def:1}
\end{definition}

\noindent Thus, given $\mathcal F^X$ we require $\{\xi_t\}$ to have conditionally independent 
and stationary increments under $\mathbb P$, and to possess a conditional exponent of the 
form
\begin{eqnarray}
\psi_{X}(\alpha) := t^{-1} \ln  {\mathbb E}^{{\mathbb P}}
[\exp(\alpha \xi_t)\, | \, {\mathcal F}^X] 
= \psi_0(\alpha + X) - \psi_0(X)
\label{eq:zz8}
\end{eqnarray}
for $\alpha \in \mathds C^{\rm I}$, where $\psi_0(\alpha)$ is the fiducial exponent of the 
specified noise type. It is implicit in the statement of Definition~\ref{def:1} that a certain 
compatibility condition holds between the message and the noise type. For any random 
variable $X$ we define its support $S_X$ to be the smallest closed set $F$ with the property 
that ${\mathbb P}(X \in F) = 1$. Then we say that $X$ is compatible with the fiducial exponent 
$\psi_0(\alpha)$ if $S_X \subset A$. Intuitively speaking, the compatibility condition ensures 
that we can use $X$ to make a random Esscher transformation. In the theory of signal 
processing, it is advantageous to require that the variables to be estimated should be 
square integrable. This condition ensures that the conditional expectation exists and admits 
the interpretation as a best estimate in the sense of least squares. For our purpose it will 
suffice to assume throughout the paper that the information process is square integrable 
under ${\mathbb P}$. This in turn implies that $\psi'(X)$ is square integrable, and that 
$\psi''(X)$ is integrable. Note that we do not require that the L\'evy information process 
should possess exponential moments under ${\mathbb P}$, but a sufficient condition for 
this to be the case is that there should exist a nonvanishing real number $\epsilon$ such 
that $\lambda+\epsilon \in A$ for all $\lambda \in S_X$.

To gain a better understanding of the sense in which the information process $\{\xi_t\}$ actually
``carries'' the message $X$, it will be useful to investigate its asymptotic behaviour. We write 
$I_0(y)$ for the inverse marginal fiducial exponent.

\begin{proposition} 
Let $\{\xi_t\}$ be a L\'evy information process with fiducial exponent  $\psi_0(\alpha)$ 
and message $X$. Then for  every $\epsilon > 0$ we have
\begin{eqnarray}
\lim_{t\to\infty} {\mathbb P}\big[ | I_0(t^{-1}\xi_t) - X | \geq \epsilon\big] = 0.
\label{eq:25.1}
\end{eqnarray}
\label{prop:1}
\end{proposition}
\begin{proof}
It follows from (\ref{eq:zz8}) that $\psi'_X(0)=\psi'_0(X)$, and hence that at any time $t$ the 
conditional mean of the random variable $t^{-1}\xi_t$ is given by
\begin{eqnarray}
{\mathbb E}^{\mathbb P}\left[t^{-1} \xi_t \,| \, \mathcal F^X \right] = \psi'_0(X) .
\label{eq:22}
\end{eqnarray}
A calculation then shows that the conditional variance of $t^{-1}\xi_t$ takes the form
\begin{eqnarray}
{\rm Var}^{\mathbb P}\left[ t^{-1} \xi_t \, | \, \mathcal F^X \right] 
:= \left. {\mathbb E}^{\mathbb P}\left[ \left ( t^{-1} \xi_t -\psi'_0(X) \right)^2 \, \right | \, 
\mathcal F^X \right] =  \frac {1} {t} \,  \psi_0''(X), 
\label{eq:24}
\end{eqnarray}
which allows us to conclude that 
\begin{eqnarray}
{\mathbb E}^{\mathbb P}\left[ \left ( t^{-1} \xi_t -\psi'_0(X) \right)^2  \right] 
=  \frac {1} {t} \,  {\mathbb E}^{\mathbb P}\left[ \psi_0''(X)  \right] , 
\label{eq:24xx}
\end{eqnarray}
and hence that
\begin{eqnarray}
\lim_{t\to\infty} {\mathbb E}^{\mathbb P}\left[ \left ( t^{-1} \xi_t -\psi'_0(X) \right)^2  \right]  = 0.
\label{eq:25.2}
\end{eqnarray}
On the other hand for all $\epsilon > 0$ we have
\begin{eqnarray}
{\mathbb P}[\, | t^{-1}\xi_t - \psi'_0(X) | \geq \epsilon\, ]
\leq \frac {1}{\epsilon^2} \, {\mathbb E}^{\mathbb P}\left[ \left ( t^{-1} \xi_t -\psi'_0(X) \right)^2 
\right]
\label{eq:25.3}
\end{eqnarray}
by Chebychev's inequality, from which we deduce that
\begin{eqnarray}
\lim_{t\to\infty} \mathbb P[\, | t^{-1}\xi_t - \psi'_0(X) | \geq \epsilon\, ] = 0,
\label{eq:25.4}
\end{eqnarray}
and it follows  that $ I_0(t^{-1}\xi_t)$ converges to $X$ in probability.
 \qed 
\end{proof}

Thus we see that the 
information process does indeed carry information about the message, and in the long run  
``reveals'' it. The intuition here is that as more information is gained we improve our estimate of 
$X$ to the point that the value of $X$ eventually becomes known with near certainty. 

\section{Properties of L\'evy information}
\label{sec:3}

It will be useful if we present a construction that ensures the existence of L\'evy information 
processes. First we select a noise type by specification of a fiducial characteristic $K_0$. 
Next we introduce a probability space $({\mathit\Omega},{\mathcal F},{\mathbb P}_0)$ that 
supports the existence of a ${\mathbb P}_0$-L\'evy process $\{\xi_t\}$ with the given fiducial 
characteristic, together with an independent random variable $X$ that is compatible with $K_0$.  

Write $\{{\mathcal F}_t\}$ for the filtration generated by $\{\xi_t\}$, and $\{{\mathcal G}_t\}$ 
for the filtration generated by $\{\xi_t\}$ and $X$ jointly: ${\mathcal G}_t = \sigma 
[\{\xi_t\}_{0 \leq s \leq t}, X]$.  Let $\psi_0(\alpha)$ be the fiducial exponent associated with 
$K_0$. One can check that the process $\{\rho^X_t\}$ defined by
\begin{eqnarray}
\rho^X_t = \exp\left({X \xi_t - \psi_0(X)\, t}\right) 
\label{eq:12}
\end{eqnarray}
is a $(\{{\mathcal G}_t \},{\mathbb P}_0)$-martingale. We are thus able to introduce a change 
of measure ${\mathbb P}_0 \rightarrow{\mathbb P}_X$ on $({\mathit\Omega},{\mathcal F},
{\mathbb P}_0)$ by setting
\begin{eqnarray}
\left.\frac{\rd {\mathbb P}_{\!X}}{\rd{\mathbb P}_0}\right|_{\mathcal G_t} = \rho^X_t.
\end{eqnarray}
It should be evident that $\{\xi_t\}$ is conditionally ${\mathbb P}_{\!X}$-L\'evy given 
${\mathcal F}^X$, since for fixed $X$ the measure change is an Esscher transformation. 
In particular, a calculation shows that the conditional exponent of $\xi_t$ under 
${\mathbb P}_X$ is given by
\begin{eqnarray}
t^{-1} \ln  {\mathbb E}^{{\mathbb P_X}}\left [\exp ({\alpha \xi_t})\, | \, \mathcal F^X\right ] 
= \psi_0(\alpha + X) - \psi_0(X)
\end{eqnarray}
for $\alpha \in \ \mathds C^{\rm I}$, which shows that the conditions of Definition~\ref{def:1} 
are satisfied, allowing us to conclude the following:
%
\begin{proposition} 
The $\mathbb P_{0}$-L\'evy process $\{\xi_t\}$ is a ${\mathbb P}_{\!X}$-L\'evy information process,
 with  message $X$ and noise type $\psi_0(\alpha)$.  
\label{prop:2}
\end{proposition}

In fact, the converse also holds: if we are given a L\'evy information process, then by a change of 
measure we can find a L\'evy process and an independent ``message'' variable. Here follows a 
more precise statement. 

\begin{proposition}
Let $\{\xi_t\}$ be a L\'evy information process on a probability space 
$({\mathit\Omega},{\mathcal F}, {\mathbb P})$ with message 
$X$ and noise type $\psi_0(\alpha)$. Then there exists a change of measure ${\mathbb P}\to
{\mathbb P}_0$ such that $\{\xi_t\}$ and $X$ are ${\mathbb P}_0$-independent, 
$\{\xi_t\}$ is ${\mathbb P}_0$-L\'evy with exponent $\psi_0(\alpha)$, and the probability law of 
$X$ under ${\mathbb P}_0$ 
is the same as probability law of $X$ under ${\mathbb P}$. 
\label{prop:3}
\end{proposition}

\begin{proof}
First we establish that the process $\{{\tilde\rho}^X_t\}$ defined by the expression 
${\tilde\rho}^X_t=\exp(-X\xi_t+\psi_0(X)t)$ is a $(\{{\mathcal G}_t \},{\mathbb P})$-martingale. 
We have
\begin{eqnarray}
{\mathbb E}^{\mathbb P}[{\tilde\rho}^X_t|{\mathcal G}_s] 
&=& {\mathbb E}^{\mathbb P}\left[ \exp(-X\xi_t+\psi_0(X)t) \,|\,{\mathcal G}_s\right] 
\nonumber \\ 
&=& {\mathbb E}^{\mathbb P}[\exp(-X(\xi_t-\xi_s))|{\mathcal G}_s] 
\exp(-X\xi_s+\psi_0(X)t) \nonumber \\ &=& \exp( \psi_X(-X)(t-s)) \exp(-X\xi_s+\psi_0(X)t) 
\end{eqnarray}
by virtue of the fact that $\{\xi_t\}$ is  $\mathcal F^X$-conditionally L\'evy under 
$\mathbb P$. By use of (\ref{eq:zz8}) we deduce that $\psi_X(-X)=-\psi_0(X)$, and hence that 
${\mathbb E}^{\mathbb P}[{\tilde\rho}^X_t|{\mathcal G}_s] =  {\tilde \rho}^X_s$, as required. 
Then we use $\{{\tilde\rho}^X_t\}$ to define a change of measure  $\mathbb P  \rightarrow 
{\mathbb P}_0$ on  $({\mathit\Omega},{\mathcal F}, {\mathbb P})$ by setting
\begin{eqnarray}
\left.\frac{\rd {\mathbb P}_{0}}{\rd{\mathbb P}}\right|_{\mathcal G_t} = {\tilde\rho}^X_t.
\end{eqnarray}
To show that $\xi_t$ and $X$ are ${\mathbb P}_0$-independent for all $t$, it suffices to 
show that their joint characteristic function under ${\mathbb P}_0$ factorises. Letting  
$\alpha, \beta \in \ \mathds C^{\rm I}$, we have 
\begin{eqnarray}
{\mathbb E}^{{\mathbb P}_0}[\exp(\alpha \xi_t+  \beta X)] 
&=& {\mathbb E}^{\mathbb P} 
\left[\exp(-X\xi_t+\psi_0(X)t) \exp(\alpha\xi_t+  \beta X)\right] \nonumber \\ 
&=& 
{\mathbb E}^{\mathbb P}\left[{\mathbb E}^{\mathbb P}[\exp((-X+\alpha)\xi_t+\psi_0(X)t+
\beta X)|{\mathcal F}^X]\right] \nonumber \\ 
&=& {\mathbb E}^{\mathbb P}[ 
\exp(\psi_X(-X+\alpha)t+\psi_0(X)t+  \beta X)] \nonumber \\ 
&=& 
\exp( \psi_0(\alpha)t)\, {\mathbb E}^{{\mathbb P}}[\exp(  \beta X)]  ,
\label{factorization}
\end{eqnarray}
where the last step follows from (\ref{eq:zz8}). This argument can be extended to show that 
$\{\xi_t\}$ and $X$ are ${\mathbb P}_0$-independent. Next we observe that 
\begin{eqnarray}
{\mathbb E}^{{\mathbb P}_0}[\exp(\alpha ( \xi_u - \xi_t )+  \beta \xi_t)] 
&& \nonumber \\ && \hspace{-3.0cm} = \, {\mathbb E}^{\mathbb P} 
\left[\,\exp(-X\xi_u+\psi_0(X)u +  \alpha ( \xi_u - \xi_t )+  \beta \xi_t )\,\right] \nonumber \\ 
&& \hspace{-3.0cm} = \, 
{\mathbb E}^{\mathbb P}\left[{\mathbb E}^{\mathbb P}[\exp(-X\xi_u+\psi_0(X)u +  \alpha 
( \xi_u - \xi_t )+  \beta \xi_t )\,|{\mathcal F}^X]\right] \nonumber \\ 
&& \hspace{-3.0cm} =  \,
{\mathbb E}^{\mathbb P}\left[{\mathbb E}^{\mathbb P}[\exp(\psi_0(X)u +  (\alpha - X) 
(\xi_u - \xi_t )+  (\beta - X) \xi_t )\,|{\mathcal F}^X]\right] \nonumber \\ 
&& \hspace{-3.0cm} = \,  {\mathbb E}^{\mathbb P}[ 
\exp( \psi_0(X)u + \psi_X(\alpha - X)(u - t)+  \psi_X(\beta - X)t ) \, ] \nonumber \\ 
&& \hspace{-3.0cm} = \, \exp( \psi_0(\alpha)(u-t))\, \exp( \psi_0(\beta)t) 
\label{ind incr}
\end{eqnarray}
for $u \geq t \geq 0$, and it follows that $\xi_u - \xi_t $ and  $\xi_t $ are independent. This
argument  can be extended to show that $\{\xi_t\}$ has ${\mathbb P}_0$-independent 
increments. Finally, if we set $\alpha = 0$ in (\ref{factorization}) it follows that the probability 
laws of $X$ under ${\mathbb P}_0$ and ${\mathbb P}$ are identical; if we set $\beta=0$ in 
(\ref{factorization}) it follows that the ${\mathbb P}_0$ exponent of $\{\xi_t\}$ is 
$\psi_0(\alpha)$; and if we set $\beta = 0$ in (\ref{ind incr}) it follows that $\{\xi_t\}$ is 
${\mathbb P}_0$-stationary. 
 \qed 
\end{proof}

Going forward, we adopt the convention that  ${\mathbb P}$ always denotes  the ``physical'' 
measure in relation to which an information process with message $X$ is defined, and that 
${\mathbb P_0}$ denotes the  transformed measure with respect to which the information 
process and the message decouple. Therefore, henceforth we write ${\mathbb P}$ rather 
than ${\mathbb P}_X$. In addition to establishing the existence of L\'evy information 
processes, the results of Proposition~\ref{prop:3} provide useful tools for calculations, allowing 
us to work out properties of information processes by referring the calculations back to 
${\mathbb P}_0$. We consider as an example the problem of working out the 
${\mathcal F}_t$-conditional expectation under ${\mathbb P}$ of a 
${\mathcal G}_t$-measurable integrable random variable $Z$. The ${\mathbb P}$-expectation 
of $Z$ can be written in terms of ${\mathbb P}_{0}$-expectations, and is given by a ``generalised 
Bayes formula'' (Kallianpur \& Striebel 1968) of the form
\begin{eqnarray}
{\mathbb E}^{{\mathbb P}} \! \left[ Z \, |\, {\mathcal F}_t\right] = 
\frac{{\mathbb E}^{\mathbb  P_{0}}\! \left[ \rho^X_t Z \, |\,  {\mathcal F}_t\right]}
{{\mathbb E}^{\mathbb P_{0}}\! \left[ \rho^X_t \, |\,  {\mathcal F}_t\right]} . 
\label{eq:zz14}
\end{eqnarray}
This formula can be used to obtain the ${\mathcal F}_t$-conditional probability distribution 
function  for $X$, defined for 
$y\in{\mathds R}$ by 
\begin{eqnarray}
F^X_t (y) = \mathbb P (X\leq y \,|\,{\mathcal F}_t). 
\end{eqnarray}
In the Bayes formula we set  $Z = {\mathds 1}\{X\leq y\}$, and the result is
\begin{eqnarray}
F^X_t(y) = 
\frac{\int {\mathds 1}\{x\leq y\} 
\exp \left({x\xi_t - \psi_0(x) t}\right) \, \rd F^X(x)}
{\int \exp  \left({x\xi_t - \psi_0(x) t}\right) \, \rd F^X(x)} ,
\label{eq:zz20}
\end{eqnarray}
where $F^X(y)= {\mathbb P}(X<y)$ is the \textit{a priori} distribution function. It is useful for 
some purposes to work directly with the conditional probability measure $\pi_t(\rd x)$ induced 
on ${\mathds R}$ defined by $\rd F^X_t(x)=   \pi_t(\rd x)$. In particular, when $X$ is a 
continuous random variable with a density function $p(x)$ one can write $\pi _t(\rd x) = p_t(x)
\rd x$, where $p_t(x)$ is the conditional density function. 

\begin{proposition} 
Let $\{\xi_t\}$ be a L\'evy information process under ${\mathbb P}$ with noise type 
$\psi_0(\alpha)$, and let the {\rm a priori} distribution of the associated message $X$ be
$ \pi (dx)$. Then the ${\mathcal F}_t$-conditional {\rm a posteriori} distribution of $X$ 
is  
\begin{eqnarray}
\pi_t( \rd x) =  \frac{ \, \exp \left ({x\xi_t - \psi_0(x) t}\right)  }
{\int  \exp \left ({x\xi_t - \psi_0(x) t}\right) \pi (\rd x) } \, \pi (\rd x) . 
\label{eq:zz18}
\end{eqnarray}
\label{prop:4}
\end{proposition} 

It is straightforward to establish by use of a variational argument that for any function 
$f: \, {\mathds R}\to{\mathds R}$ such that the random variable $Y=f(X)$ is integrable, 
the best estimate for $Y$ conditional on the information ${\mathcal F}_t$ is given by 
\begin{eqnarray}
{\hat Y}_t :={\mathbb E}^{{\mathbb P}} [Y \,|\, {\mathcal F}_t] = \int f(x) \, \pi_t( \rd x) . 
\label{best estimate}
\end{eqnarray}
By the  ``best estimate'' for $Y$ we mean the ${\mathcal F}_t$-measurable random variable 
${\hat Y}_t$ that minimises the quadratic error ${\mathbb E}^{{\mathbb P}}[(Y-{\hat Y}_t)^2
|{\mathcal F}_t]$. 

It will be observed that at any given time $t$ the best estimate can be 
expressed as a function of $\xi_t$ and $t$, and does not involve values of the information 
process at times earlier than $t$. That this should be the case can be seen as a consequence 
of the following: 

\begin{proposition} 
The L\'evy information process $\{\xi_t\}$ has the Markov property.  
\label{prop:5}
\end{proposition} 
\begin{proof}
For the Markov property it suffices to establish that for $a \in{\mathds R}$ we have
\begin{eqnarray}
{\mathbb P}\left( \xi_t\leq a \, | \, {\mathcal F}_s \right) = {\mathbb P}\left( \xi_t\leq a\, | 
\, {\mathcal F}^{\xi_s}  \right), 
\label{Markov property}
\end{eqnarray}
where 
${\mathcal F}_t=\sigma[\,\{\xi_s\}_{0 \leq s \leq t}\,]$ 
and 
${\mathcal F}^{\xi_t} =\sigma[\,\xi_t\,]$. 
We write
\begin{eqnarray}
\Phi_t := {\mathbb E}^{{\mathbb P}_{0}} \! 
\left[ \rho^X_t | {\mathcal F}_t\right] 
= \int   \exp \left ({x\xi_t - \psi_0(x) t}\right) \,  \pi( \rd x), 
\label{eq:9}
\end{eqnarray}
where $\rho^X_t$ is defined as in equation (\ref{eq:12}). It follows that 
\begin{eqnarray}
&& \hspace{-1.0cm} {\mathbb P}\left( \xi_t\leq a \,|\, {\mathcal F}_s \right) =
{\mathbb E}^{\mathbb P}[{\mathds 1}\{\xi_t\leq a\}\,|\, {\mathcal F}_s]
= \frac{{\mathbb E}^{{\mathbb P}_{0}}[\Phi_t{\mathds 1}\{\xi_t\leq a\}
\,|\, {\mathcal F}_s]}{{\mathbb E}^{{\mathbb P}_{0}}[\Phi_t \,|\, {\mathcal F}_s]} 
\nonumber \\ && \hspace{-0.3cm}  =  
\frac{{\mathbb E}^{{\mathbb P}_{0}}[\Phi_t{\mathds 1}\{\xi_t\leq a\}
\,|\, {\mathcal F}^{\xi_s} ]}{{\mathbb E}^{{\mathbb P}_{0}}[\Phi_t \,|\, {\mathcal F}^{\xi_s} ]} 
= {\mathbb E}^{\mathbb P}[{\mathds 1}\{\xi_t\leq a\} \,|\, {\mathcal F}^{\xi_s}  ] 
= {\mathbb P}\left( \xi_t\leq a\, | \, {\mathcal F}^{\xi_s}  \right),
\end{eqnarray}
since $\{\xi_t\}$ has the Markov property under the transformed measure ${\mathbb P}_0$.  
 \qed 
\end{proof}

We note that since $X$ is ${\mathcal F}_\infty$-measurable, which follows from 
Proposition~\ref{prop:1}, the Markov property implies that if 
$Y=f(X)$ is integrable we have 
\begin{eqnarray}
{\mathbb E}^{\mathbb P} [Y|{\mathcal F}_t] = {\mathbb E}^{\mathbb P}
[Y | \,{\mathcal F}^{\xi_t} ].
\label{Markov consequence}
\end{eqnarray}
This identity allows one to work out the optimal filter for a L\'evy information process 
by direct use of the Bayes formula. It should be apparent that simulation of the dynamics of the 
filter is readily approachable on account of this property.

We remark briefly on what might appropriately be called a ``time consistency'' property satisfied 
by L\'evy information processes. It follows  from (\ref{eq:zz18}) that, given the conditional 
distribution $\pi_s( \rd x)$ at time $s\leq t$, we can express $\pi_t( \rd x)$ in the form  
\begin{eqnarray}
\pi_t( \rd x) = \frac{  \exp \big( x(\xi_t-\xi_s) - \psi_0(x)(t-s)\big)}
{\int   \exp \big( x(\xi_t-\xi_s) - \psi_0(x)(t-s)\big)\pi_s( \rd x)} \, \pi_s(\rd x) .
\label{eq:ww1}
\end{eqnarray}
Then if for fixed $s \geq 0$ we introduce a new time variable $u := t - s$, and define 
$\eta_u =\xi_{u + s} - \xi_{ s}$, we find that $\{ \eta_u \}_{u \geq 0}$ is an information process 
with fiducial exponent $\psi_0(\alpha)$ and message $X$ with \textit {a priori} distribution 
$\pi_s( \rd x)$. Thus given up-to-date information we can ``re-start'' the information process 
at that time to produce a new information process of the same type, with an adjusted 
message distribution. 

Further insight into the nature of L\'evy information can be gained by examination 
of expression (\ref{eq:zz8}) for the conditional exponent of an information process. In 
particular, as a consequence of the L\'evy-Khintchine representation (\ref{eq:2}) we are able to deduce 
that  
\begin{eqnarray}
\psi_0(\alpha + X) - \psi_0(X)\!\! &=& \!\! \left( p + qX +  \int_{{\mathds R} \backslash \{0\}}
\!\!\!\!\! z (\re^{Xz} -1){\mathds 1}\{|z|<1\}) \nu(\rd z) \right) \alpha 
+ \half q \alpha^2 \nonumber \\ && + \int_{{\mathds R} \backslash \{0\}} \!\!\!\!\!  
(\re^{\alpha z} -1 -\alpha z {\mathds 1}\{|z|<1\}) \re^{Xz} \nu(\rd z),
\label{eq:xx16}
\end{eqnarray}
for $\alpha \in \mathds C^{\rm I}$, which leads to the following: 

\begin{proposition}
The randomisation of the ${\mathbb P}_0$-L\'evy process $\{\xi_t\}$ achieved through the 
change of measure generated by the randomised Esscher martingale 
$\rho_t=\exp(X \xi_t - \psi_0(X)t)$ induces two effects on the characteristics of the process: 
{\rm (i)} a random shift in the drift term, given by
\begin{eqnarray}
p \to p + qX +  \int_{{\mathds R} \backslash \{0\}}
z (\re^{Xz} -1){\mathds 1}\{|z|<1\}) \nu(\rd z)  , 
\label{eq:xx16.1}
\end{eqnarray}
and {\rm (ii)} a random rescaling of the L\'evy measure, given by $\nu(\rd z) \to \re^{Xz} 
\nu(\rd z)$.
\label{prop:6}
\end{proposition} 

The integral appearing in  the shift in the drift term is well 
defined since the term $z (\re^{Xz} -1)$ vanishes to second order at the origin. It follows from 
Proposition~\ref{prop:6} that in sampling an information process an agent is in 
effect trying to detect a random shift in the drift term, and a random ``tilt'' and 
change of scale in the L\'evy measure, altering the overall rate as well as the relative rates 
at which jumps of various sizes occur. It is from these data, within which the message is 
encoded,  that the agent attempts to estimate the value of $X$. It is interesting to note that 
randomised Esscher martingales arise in the construction of pricing kernels in the theory of 
finance (see, e.g., Shefrin 2008, Macrina \& Parbhoo 2011). 

We turn to examine the properties of certain martingales associated with L\'evy information. 
We establish the existence of a so-called innovations representation for 
L\'evy information. In the case of the Brownian filter the ideas involved are rather well 
understood (see, e.g., Liptser \& Shiryaev 2000), and the matter has also been investigated 
in the case of Poisson information (Segall \& Kailath 1975). These examples 
arise as special cases in the general theory of L\'evy information. Throughout 
the discussion that follows we fix a probability space 
$({\mathit\Omega},{\mathcal F},{\mathbb P})$.

\begin{proposition}
Let $\{\xi_t\}$ be a L\'evy information process with fiducial exponent $\psi_0(\alpha)$
and message $X$, let 
$\{{\mathcal F}_t\}$ denote the filtration generated by $\{\xi_t\}$, let $Y = \psi'_0(X)$, 
where $\psi_0'(\alpha)$ is the marginal fiducial exponent, and set 
$ \hat Y_t= {\mathbb E}^{\mathbb P}\left[Y|{\mathcal F}_t \right] $. 
Then the process $\{M_t\}$ defined by 
\begin{eqnarray}
 \xi_t =  \int_0^t \hat Y_u \,  \rd u + M_t
\label{eq:59}
\end{eqnarray}
is an $(\{{\mathcal F}_t \},{\mathbb P})$-martingale.
\label{prop:7}
\end{proposition}

\begin{proof}
We recall that $\{\xi_t\}$ is by definition ${\mathcal F}^X$-conditionally ${\mathbb P}$-L\'evy. 
It follows therefore from (\ref{eq:22}) that ${\mathbb E}^{\mathbb P}[\xi_t|X]=Yt$, where 
$Y = \psi'_0(X)$. As before we let $\{{\mathcal G}_t \}$ denote the filtration generated jointly 
by $\{\xi_t\}$ and $X$. First we observe that the process defined for $t \geq 0$ by $m_t = 
\xi_t-Yt$ is a $(\{{\mathcal G}_t\} ,{\mathbb P})$-martingale. This assertion can be checked 
by consideration of  the one-parameter family of 
$(\{{\mathcal G}_t\},{\mathbb P}_0)$-martingales defined by
\begin{eqnarray}
\rho^{X+\epsilon}_t = \exp \big({(X + \varepsilon) \xi_t - \psi_0(X + \varepsilon) t}\big)
\end{eqnarray}
for $\epsilon \in \mathds C^{\rm I}$. Expanding this expression to first order in $\epsilon$,
we deduce that the process defined for $t \geq 0$ by
$ \rho^X_t (\xi_t-\psi_0'(X)t)$
is a $( \{ {\mathcal G}_t \},{\mathbb P}_0)$-martingale. Thus we have
\begin{eqnarray}
{\mathbb E}^{{\mathbb P}_{0}}\left[\rho^X_t(\xi_t-\psi_0'(X)t) \, |\, {\mathcal G}_s\right] = 
\rho^X_s (\xi_s-\psi_0'(X)s).
\label{eq:zz62}
\end{eqnarray}
\noindent
Then using $\{\rho^X_t\}$ to make a change of measure from ${\mathbb P}_0$ to ${\mathbb P}$ 
we obtain
\begin{eqnarray}
{\mathbb E}^{{\mathbb P}}\left[\xi_t-\psi_0'(X)t \, |\, {\mathcal G}_s\right] = 
\xi_s-\psi_0'(X)s , 
\end{eqnarray}
and the result follows if we set $Y = \psi_0'(X)$. 
Next we introduce the ``projected'' process  $\{{\hat m}_t\}$ defined by 
$ \hat m_t= {\mathbb E}^{\mathbb P}\left[ m_t \, |\, {\mathcal F}_t \right]$. 
We note that since $\{m_t\}$ is a $(\{{\mathcal G}_t \} ,{\mathbb P})$-martingale we have
\begin{eqnarray}
{\mathbb E}^{\mathbb P}[\hat m_t|{\mathcal F}_s] 
&=&
 {\mathbb E}^{\mathbb P}[\xi_t-Yt \,| \,{\mathcal F}_s] 
\nonumber \\ 
&=& 
{\mathbb E}^{\mathbb P}\left[ {\mathbb E}^{\mathbb P}
[\xi_t-Yt \,|\, {\mathcal G}_s]|{\mathcal F}_s\right] 
\nonumber \\
&=&
 {\mathbb E}^{\mathbb P}[\xi_s-Ys \,| \,{\mathcal F}_s] 
\nonumber \\ 
&=& \hat m_s,
\label{eq:61}
\end{eqnarray}
and thus $\{\hat m_t\}$ is an $(\{{\mathcal F}_t \} ,{\mathbb P})$-martingale. 
Finally we observe that
\begin{align}
{\mathbb E}^{\mathbb P}\left[M_t|{\mathcal F}_s\right] 
= 
{\mathbb E}^{\mathbb P}\left[\left. \xi_t - \int_0^t {\hat Y}_u \,\rd u \right| {\mathcal F}_s\right] 
=
{\mathbb E}^{\mathbb P}[\xi_t|{\mathcal F}_s]  - 
{\mathbb E}^{\mathbb P}\left[\left. \int_s^t {\hat Y}_u \,\rd u \right| {\mathcal F}_s\right]
- \int_0^s {\hat Y}_u \,\rd u,
\label{eq:62}
\end{align}
where we have made use of the fact that the final term is ${\mathcal F}_s$-measurable. 
The fact that $\{{\hat m}_t\}$ and $\{{\hat Y}_t\}$
are both $({\mathcal F}_t,{\mathbb P})$-martingales implies
that
\begin{eqnarray}
{\mathbb E}^{\mathbb P}[\xi_t|{\mathcal F}_s] - \xi_s  
= (t-s){\hat Y}_s
={\mathbb E}^{\mathbb P}\left[\left. \int_s^t {\hat Y}_u \,\rd u \right| {\mathcal F}_s\right] ,
\end{eqnarray}
from which it follows that 
${\mathbb E}^{\mathbb P}\left[\left. M_t \right| {\mathcal F}_s\right] = M_s$, which is what we 
set out to prove. 
 \qed 
\end{proof}

Although the general information process does not admit an additive decomposition into signal 
and noise, it does admit a linear decomposition into terms representing (i) information already 
received and (ii) new information. The random variable $Y$ entering via its conditional 
expectation into the first of these terms is itself in general a nonlinear function of the 
message variable $X$. It follows on account of the convexity of the fiducial exponent that 
the marginal fiducial exponent is invertible, which ensures that $X$ can be expressed in 
terms of $Y$ by the relation $X = I_0(Y)$, which is linear if and only if the information 
process is Brownian. Thus signal and noise are deeply intertwined in the case of general 
L\'evy information. Vestiges of linearity remain, and these suffice to provide an overall 
element of tractability.

\section{Examples of L\'evy information processes}
\label{sec:4}

In a number of situations one can construct explicit examples of information processes, 
categorised by noise type. The Brownian and Poisson constructions, which are familiar in 
other contexts, can be seen as belonging to a unified scheme that brings out their 
differences and similarities. We then proceed to construct information processes of the 
gamma, the variance gamma, the negative binomial, the inverse Gaussian, and the normal 
inverse Gaussian type. It is interesting to take note of the diverse nature of noise, and to 
observe the many different ways in which messages can be conveyed in a noisy 
environment.  \\ 

\noindent {\bf Example 1: Brownian information}. 
On a probability space 
$({\mathit\Omega},{\mathcal F}, {\mathbb P})$, let $\{B_t\}$ be a Brownian 
motion, let $X$ be an independent random variable, and set
\begin{eqnarray}
\xi_t = X t + B_t.  
\end{eqnarray} 
The random process $\{\xi_t\}$ thereby defined, which we call the Brownian information 
process, is $\mathcal F^X$-conditionally $K_X$-L\'evy, with conditional characteristic 
$K_{X} = (X, 1, 0)$ and conditional exponent $\psi_X(\alpha) = X \alpha + \half  \alpha^2$. 
The fiducial characteristic is $K_0$ = $(0, 1, 0)$, the fiducial exponent is $\psi_0(\alpha) = 
\half  \alpha^2$, and the associated fiducial process or ``noise type'' is standard Brownian 
motion. In the case of Brownian information, there is a linear separation of the process into 
signal and noise. This model, considered by Wonham (1965), is perhaps the simplest 
continuous-time generalisation of the example described by Wiener (1948). The message is 
given by the value of $X$, but $X$ can only be observed indirectly, through $\{\xi_t\}$. The 
observations of $X$ are obscured by the noise represented by the Brownian motion $\{B_t\}$. 
Since the signal term grows linearly in time, whereas $|B_t|\sim\sqrt{t}$, it is intuitively 
plausible that observations of $\{\xi_t\}$ will asymptotically reveal the value of $X$, and a 
direct calculation using properties of the normal distribution function confirms that 
$ t^{-1}\xi_t$ converges in probability to $X$; this is consistent with Proposition~\ref{prop:1} 
if we note that $\psi'_0(\alpha)=\alpha$ and $I_0(y)=y$ in the Brownian case. 

The best estimate for $X$ conditional on ${\mathcal F}_t$ is given by (\ref{best estimate}), 
which can be derived by use of the generalised Bayes formula (\ref{eq:zz14}). In the 
Brownian case there is an elementary method leading to the same result, worth mentioning 
briefly since it is of interest. First we present an alternative proof of Proposition~\ref{prop:5} 
in the Brownian case  that uses a Brownian bridge argument. 

We recall that if $s>s_1>0$ 
then $B_s$  and $s^{-1}B_s - s_1^{-1}B_{s_1} $ are independent. More generally, we 
observe that if $s>s_1 > s_2$, then $B_s$\,, $s^{-1}B_s - s_1^{-1}B_{s_1}$, and 
$s_1^{-1}B_{s_1}- s_2^{-1}B_{s_2}$ are independent, and that 
$s^{-1}\xi_s - s_1^{-1}\xi_{s_1} =s^{-1}B_s - s_1^{-1}B_{s_1}$. Extending this line of 
reasoning, we see that for any $a \in{\mathds R}$ we have
\begin{eqnarray}
{\mathbb P}\left( \xi_t\leq a \, | \, \xi_s,\xi_{s_1}, \ldots,\xi_{s_k}\right) 
&=& 
{\mathbb P}\left(\xi_t\leq a \, \Big| \, \xi_s, \frac{\xi_s}{s} -\frac{\xi_{s_1}}{s_1}, 
\ldots, \, \frac{\xi_{s_{k-1}}}{s_{k-1}} - \frac{\xi_{s_k}}{s_k} \right) \nonumber \\ 
&=& 
{\mathbb P}\left( \xi_t\leq a \, \Big| \, \xi_s \right),
\label{eq:4}
\end{eqnarray}
since $\xi_t$ and $\xi_s$ are independent of $s^{-1}B_s -s_1^{-1}B_{s_1}$, $\ldots$\,, 
$\,s_{k-1}^{-1}B_{s_{k-1}} -s_k^{-1}B_{s_k}$, and that gives us the Markov property 
(\ref{Markov property}). Since we have established that $X$ is ${\mathcal F}_\infty$-measurable, 
it follows that (\ref{Markov consequence}) holds. As a consequence, the \textit{a posteriori} 
distribution of $X$ can be worked out by use of the standard Bayes formula, and for the best 
estimate of $X$ we obtain  
\begin{eqnarray}
{\hat X}_t = \frac{ \int x \,  \exp ({x\xi_t - \frac{1}{2}x^2 t})\, \pi (\rd x)}
{\int \exp ({x\xi_t - \frac{1}{2}x^2 t})\, \pi (\rd x)} .
\label{eq:zz28}
\end{eqnarray}

The innovations representation (\ref{eq:59}) in the case of a Brownian information process 
can be derived by the following argument. We observe that the 
$(\{\mathcal F_t \},{\mathbb P}_0)$-martingale $\{\Phi_t\}$ defined in (\ref{eq:9}) is a 
``space-time'' function of the form
\begin{eqnarray}
\Phi_t := {\mathbb E}^{{\mathbb P}_{0}} \!\left[ \rho_t \,|\, {\mathcal F}_t\right] = 
\int  \exp \left({x\xi_t - \frac{1}{2}x^2 t}\right)\, \pi (\rd x).
\label{eq:zz29}
\end{eqnarray}
By use of the Ito calculus together with (\ref {eq:zz28}), we deduce that 
$\rd \Phi_t = {\hat X}_t \, \Phi_t \,\rd \xi_t$, and thus by integration we obtain
\begin{eqnarray}
\Phi_t = \exp\left(\int_0^t {\hat X}_s {\rm d}\xi_s - \frac{1}{2}\int_0^t {\hat X}_s^2 
{\rm d}s\right) . 
\label{exponential martingale}
\end{eqnarray}
Since $\{\xi_t\}$ is an $(\{\mathcal F_t\},{\mathbb P}_0)$-Brownian motion, it follows from 
(\ref{exponential martingale}) by the Girsanov theorem 
that the process $\{M_t\}$ defined by 
\begin{eqnarray}
\xi_t  =  \int_0^t {\hat X}_s \, {\rd}s + M_t 
\label{eq:14}
\end{eqnarray}
is an $(\{\mathcal F_t\},{\mathbb P})$-Brownian motion, which we call the innovations 
process (see, e.g., Heunis 2011). The increments of $\{M_t\}$ represent 
the arrival of new information. 

We conclude our discussion of Brownian information with the following remarks. In problems 
involving prediction and valuation, it is not uncommon that the message is revealed after the 
passage of a finite amount of time. This is often the case in applications to finance, where the 
message takes the form of a random cash flow at some future date, or, more generally, a 
random factor that affects such a cash flow. There are also numerous examples coming from 
the physical sciences, economics and operations research where the goal of an agent is to 
form a view concerning the outcome of a future event by monitoring the flow of information 
relating to it. How does one handle problems involving the revelation of information over finite time horizons? 

One way of modelling finite time horizon scenarios in the present context is by use of a time 
change. If $\{\xi_t\}$ is a L\'evy information process with message $X$ and a specified 
fiducial exponent, then a generalisation of Proposition~\ref{prop:1} shows that the process 
$\{\xi_{tT}\}$ defined over the time interval $0 \leq t<T$ by
\begin{eqnarray}
\xi_{tT} =  \frac {T - t} {T} \,  \xi  \left( \frac {tT} {T-t} \right)
\label{finite time process}
\end{eqnarray}
reveals the value of $X$ in the limit as $t \rightarrow T$, and one can check that
\begin{eqnarray}
{\rm Cov} \left[ \xi_{sT}, \xi_{tT}  \,|\, \mathcal F^X \right] =  \frac {s(T - t)} {T} \,  
\psi''_0(X), \quad (0 \leq s \leq t < T).
\end{eqnarray}
In the case where $\{\xi_{t}\}$ is a Brownian information process represented as above in 
the form $\xi_t = Xt + B_t $, the time-changed process (\ref{finite time process}) takes the 
form $\xi_{tT} =Xt +  \beta_{tT} $, where $\{\beta_{tT}\}$ is a Brownian bridge over the 
interval $[0,T]$. Such processes have had applications in physics (Brody \& Hughston 2005, 
2006; see also Adler \textit{et al}.~2001, Brody \& Hughston 2002) and in finance 
(Brody \textit{et al}.~2007, 2008a, Rutkowski \& Yu 2007, Brody \textit{et al}.~2009, Filipovi\'c 
\textit{et al}.~2012). It seems reasonable to conjecture that time-changed L\'evy 
information processes of the more general type proposed above may be similarly applicable. 

\vspace{0.3cm}
\noindent {\bf Example 2: Poisson information}. 
Consider a situation in which an agent observes a series of events taking place at a random 
rate, and the agent wishes to determine the rate as best as possible since its value 
conveys an important piece of information. One can model the information flow in this situation 
by a modulated Poisson process for which the jump rate is an independent random 
variable. Such a scenario arises in many real-world situations, and has been investigated in 
the literature (Segall \& Kailath 1975, Segall \textit{et al}.~1975, Br\'emaud 1981, Di Masi \& 
Runggaldier 1983, Kailath \& Poor 1998). The Segall-Kailath scheme can be seen to emerge 
 naturally as an example of our general model for L\'evy information.  

As in the Brownian case, one can construct the relevant information process directly. 
On a probability space $({\mathit\Omega},{\mathcal F}, {\mathbb P})$, let 
$\{N(t) \}_{t \geq 0}$ be a standard Poisson process with jump rate $m>0$, let $X$ be an 
independent random variable, and set
\begin{eqnarray}
\xi_t = N(\re^X t).
\label{Poisson information}
\end{eqnarray}
Thus $\{\xi_t\}$  is a time-changed Poisson process, and the effect of the signal is to randomly 
modulate the rate at which the process jumps.  It is evident that $\{\xi_t\}$ is 
$\mathcal F^X$-conditionally L\'evy and satisfies the conditions of Definition~\ref{def:1}. In 
particular, 
\begin{eqnarray}
\mathbb E \left[\exp \left( {\alpha N(\re^X t) } \right)\, | \, \mathcal F^X \right] = 
\exp\left( m\re^X (\re^{\alpha} - 1 )\, t \right),
\end{eqnarray}
and for fixed $X$ one obtains a Poisson process with rate $m \re^X$. It follows that 
(\ref{Poisson information}) is an information process.  The fiducial characteristic is given by 
$K_0$ = $(0, 0, m\delta_1({\rd}z))$, that of a Poisson process with unit jumps at the rate $m$, 
where $\delta_1({\rd}z)$ is the Dirac measure with unit mass at $z=1$, and the fiducial 
exponent is $\psi_0(\alpha) = m({\re}^{\alpha} -1)$. A calculation using (\ref{eq:zz8}) shows that 
$K_X = (0,0,m \re^{X} \delta_1({\rd}z))$, and that $\psi_X(\alpha) =m \re^{X}(\re^{\alpha}-1)$. 
The relation between signal and noise in the case of Poisson information is 
rather subtle. The noise is associated with the random fluctuations of the inter-arrival times of 
the jumps, whereas the message determines the average rate at which the jumps occur. 

It will be instructive in this example to work out the conditional distribution of $X$ by elementary 
methods. Since $X$ is ${\mathcal F}_{\infty}$-measurable and $\{\xi_t\}$ has the Markov property, 
we have
\begin{eqnarray}
F^X_t (y) := \mathbb P (X\leq y \,|\, \mathcal F_t) = \mathbb P (X\leq y \,|\, \mathcal \xi_t)
\end{eqnarray}
for $y\in{\mathds R}$. It follows then from the Bayes law for an information process taking
values in ${\mathds N}_0$ that
 \begin{eqnarray}
\mathbb P (X\leq y \,|\, \xi_t = n)
 = 
\frac{\int {\mathds 1}\{x\leq y\} 
\mathbb P (\xi_t = n \, | \, X = x) \, \rd F^X(x)}
{\int \mathbb P (\xi_t = n \, | \, X = x) \, \rd F^X(x)} .
\label{discrete Bayes}
\end{eqnarray}
In the case of Poisson information the relevant conditional distribution is
\begin{eqnarray}
\mathbb P (\xi_t = n \, | \, X = x) = \exp (-mt \re^x ) \frac { (mt \re^x )^n}{n!} . 
\label{Poisson distribution}
\end{eqnarray}
After some cancellation we deduce that
\begin{eqnarray}
\mathbb P (X\leq y \,|\, \xi_t = n)
 = 
\frac{\int {\mathds 1}\{x\leq y\} 
 \exp (xn-m (\re^x -1) t) \, \rd F^X(x)}
{\int \exp (xn-m(\re^x -1) t) \, \rd F^X(x)} ,
\end{eqnarray}
and hence 
\begin{eqnarray}
F^X_t (y) 
 = 
\frac{\int {\mathds 1}\{x\leq y\} 
 \exp (x  \xi_t -m (\re^x -1) t) \, \rd F^X(x)}
{\int \exp (x  \xi_t -m (\re^x -1) t) \, \rd F^X(x)} ,
\end{eqnarray}
and thus
\begin{eqnarray}
\pi_t (\rd x) =  \frac{\exp (x  \xi_t -m (\re^x -1) t) }
{\int \exp(x\xi_t -m (\re^x -1) t) \, \pi(\rd x)} \pi(\rd x),
\label{eq:w56}
\end{eqnarray}
which we can see is consistent with (\ref{eq:zz18}) if we recall that in the case of noise of the 
Poisson type the fiducial exponent is given by $\psi_0(\alpha) = m({\re}^{\alpha} -1)$. 

If a Geiger counter is monitored continuously in time, the sound that it produces provides a nice example of a Poisson information process. The crucial message (proximity to radioactivity) carried by the noisy sputter of the instrument is represented by the rate at which the clicks occur. 

\vspace{0.3cm}
%
\noindent {\bf Example 3: Gamma information}. 
It will be convenient first to recall a few definitions and conventions (cf. Yor 2007, Brody 
\textit{et al.} 2008b, Brody \textit {et al.} 2012).  Let $m$ and $\kappa$ be positive numbers. 
By a  gamma process with rate $m$  and scale $\kappa$ on a probability space 
$({\mathit\Omega},{\mathcal F},{\mathbb P})$ we mean a L\'evy process 
$\{\gamma_t\}_{t \geq 0}$ with exponent
\begin{eqnarray}
t^{-1} \ln {\mathbb E}^{\mathbb P}\left[\exp(\alpha \gamma_t)\right] = 
-m \ln (1 - \kappa \alpha) 
\label{eq:z54}
\end{eqnarray}
for $\alpha \in A_ {\mathds C} = \{w \in {\mathds C}\, |\, {\rm Re}\,w < \kappa^{-1}\}$. 
The probability density for $\gamma_t$ is
\begin{eqnarray}
{\mathbb P}({\mathit\gamma}_t \in \rd x) = {\mathds 1}\{x>0\}
\frac{ \kappa ^{-mt} x^{mt-1} \exp{(-x/\kappa )} } { \Gamma[mt]} \, \rd x, 
\end{eqnarray}
where $\Gamma[a]$ is the gamma function. A short calculation making use of the functional 
equation $\Gamma[a + 1] = a \Gamma[a ]$ shows that 
${\mathbb E}^{\mathbb P}\,[\gamma_t] = m\kappa t $ and 
${\rm Var}^{\mathbb P}\,[\gamma_t] = m\kappa^2 t. $ Clearly, the mean and variance 
determine the rate and scale. If $\kappa = 1$ we say that $\{\gamma_t\}$ is a 
\textit{standard} gamma process with rate $m$. 
If $\kappa\neq1$ we say that $\{ \gamma_t\}$ is a scaled gamma process.
The L\'evy measure associated with the gamma process is 
\begin{eqnarray}
\nu( {\rd} z) =  {\mathds 1}\{z>0\}\, m\, z^{-1}  \exp ({-\kappa z})\, {\rd} z.
\end{eqnarray}
It follows that $\nu ({\mathds R} \backslash \{0\} ) = \infty$ and hence that the gamma 
process has infinite activity.  
Now let $\{\xi_t\}$ be a standard gamma process with rate  $m$ on a probability space 
$({\mathit\Omega},{\mathcal F},{\mathbb P}_0)$, and let $\lambda \in \mathds R$ satisfy 
$\lambda<1$. Then the process $\{\rho_t^\lambda\}$ defined by 
\begin{eqnarray}
\rho_t^\lambda = (1-\lambda)^{mt}\re^{\lambda\gamma_t}
\end{eqnarray}
is an $(\{{\mathcal F}_t\},{\mathbb P}_0)$-martingale. If we let $\{\rho_t^\lambda\}$ act as 
a change of measure density for the transformation ${\mathbb P}_0\to{\mathbb P}_\lambda$, 
then we find that $\{\gamma_t\}$ is a \textit{scaled} gamma process under 
${\mathbb P}_\lambda$, with rate $m$ and scale $1/(1-\lambda)$. Thus we see that the effect 
of an Esscher transformation on a gamma process is to alter its scale. With these facts in 
mind, one can establish the following: 

\begin{proposition}
Let $\{\gamma_t\}$ be a standard gamma process with rate $m$ on a probability space
$({\mathit\Omega},{\mathcal F},{\mathbb P})$, and let the independent random 
variable $X$ satisfy $X<1$ almost surely. Then the process $\{\xi_t\}$ defined by
\begin{eqnarray}
\xi_t = \frac{1}{1-X}\, \gamma_t 
\label{eq:z57}
\end{eqnarray}
is a L\'evy information process with message $X$ and gamma noise, with fiducial exponent 
$\psi_0(\alpha) =  -m \ln (1 - \alpha)$ for 
$\alpha \in  \{w \in {\mathds C}\, |\, {\rm Re}\,w < 1\}$. 
\label{prop:8}
\end{proposition}

\begin{proof}
It is evident that $\{\xi_t\}$ is $\mathcal F^X $-conditionally a scaled gamma process. As a 
consequence of (\ref{eq:z54}) we have 
\begin{eqnarray}
\frac{1}{t} \ln {\mathbb E}^{\mathbb P}\left[\exp(\alpha \xi_t)|X\right]  = 
\frac{1}{t} \ln {\mathbb E}^{\mathbb P}\left[\left.\exp\left(\frac{\alpha\gamma_t}{1-X}
\right) \right| X\right]  = \ln \left(1-\frac{\alpha}{1-X}\right)^{-m}
\end{eqnarray}
for $\alpha \in \mathds C^{\rm I}$. Then we note that
\begin{eqnarray}
-m \ln \left(1-\frac{\alpha}{1-X}\right) =
-m \ln \left(1-(X + \alpha) \right)  \,+\, m \ln \left(1-X \right). 
\label{eq:z58}
\end{eqnarray}
It follows that the $\mathcal F^X $-conditional $ {\mathbb P}$ exponent of $\{\xi_t\}$ 
is $\psi_0(X+\alpha)-\psi_0(X)$. 
 \qed 
\end{proof}

The gamma filter arises as follows. An agent observes a process of accumulation. Typically 
there are many small increments, but now and then there are large increments. 
The rate at which the process is growing is the figure that the agent wishes to estimate
as accurately as possible.  The accumulation can be 
modelled by gamma information, and the associated filter can be used to estimate the 
growth rate. It has long been recognised that the gamma process is useful in describing 
phenomena such as the water level of a dam or the totality of the claims made in a  
large portfolio of insurance contracts (Gani 1957, Kendall 1957, Gani \& Pyke 1960). Use of the 
gamma information process and related bridge processes, with applications in finance and 
insurance, is pursued in Brody \textit{et al}.~(2008b), Hoyle (2010), and Hoyle 
\textit{et al}.~(2011). 
We draw the reader's attention to Yor (2007) and references cited 
therein, where it is shown how certain additive properties of  Brownian motion have 
multiplicative analogues in the case of the gamma process. 
One notes in particular the remarkable 
property that $\gamma_t$ and $\gamma_s / \gamma_t$ are independent for $t\ge s \ge 0$. 
Making use of this relation, it will be instructive to present an alternative derivation of the 
optimal filter for gamma noise. We begin by establishing that  the process 
defined by (\ref{eq:z57}) has the Markov property. 
We observe first that for any times 
$t\ge s\ge s_1\ge s_2\ge\cdots\ge s_k$ the variables $\gamma_{s_1}/\gamma_{s}, 
\gamma_{s_2}/\gamma_{s_1}$, and so on,  are independent of one another and are independent 
of $\gamma_s$ and $\gamma_t$. It follows that 
\begin{eqnarray}
{\mathbb P} \left( \xi_t < a | \xi_s, \xi_{s_1}, \ldots,\xi_{s_k} \right) 
&=& 
{\mathbb P} \left( \xi_t < a |  (1-X)^{-1} \gamma_s, \ldots,
(1-X)^{-1}\gamma_{s_k}\right) 
\nonumber \\ 
&=& 
{\mathbb P} \left( \xi_t < a \left| (1-X)^{-1}\gamma_s, \, \frac{\gamma_{s_1}}{\gamma_s}, \,
\frac{\gamma_{s_2}}{\gamma_{s_1}},\,\ldots\,, \frac{\gamma_{s_k}}{\gamma_{s_{k-1}}}
\right. \right) 
\nonumber \\ 
&=& 
{\mathbb P} \left( \xi_t < a \left| \xi_s
\right. \right),
\end{eqnarray}
since $\{ \gamma_t \} $ and $X$ are independent, and that gives us  
(\ref{Markov property}). In working out the distribution of $X$ given 
${\mathcal F}_t$ it suffices therefore to work out the distribution of $X$ given 
$\xi_t$. We note that the Bayes formula implies that 
\begin{eqnarray}
\pi_t ( \rd x) = \frac{ \rho (\xi_t | X=x)}{\int \rho (\xi_t | X=x)\,\pi ( \rd x)} \, \pi ( \rd x) , 
\end{eqnarray}
where $\pi ( \rd x)$ is the unconditional distribution of $X$, and  $\rho (\xi | X=x) $ is the 
conditional density for the random variable $\xi_t$, which can be calculated as follows: 
\begin{eqnarray}
\rho (\xi | X=x) &=& \frac{\rd}{\rd \xi}\, {\mathbb P} (\xi_t \leq \xi |X=x) = 
\frac{\rd}{\rd \xi}\, {\mathbb P} ( (1-X)^{-1} \gamma_t \leq \xi | X=x) 
\nonumber \\ 
&& \hspace{-1.5cm} =
\frac{\rd}{\rd \xi}\, {\mathbb P} \left( \left.\gamma_t \leq (1-X) \xi \right| X = x \right) 
=
\frac{ \xi^{mt-1} (1-x)^{mt}  \re^{-(1-x)\xi} } {\Gamma\left[mt\right]} . 
\end{eqnarray} 
It follows that the optimal filter in the case of gamma noise is given by
\begin{eqnarray}
\pi_t (\rd x) = \frac{(1-x)^{mt}\exp({x\xi_t})}{\int_{-\infty}^1 (1-x)^{mt} 
\exp ({x\xi_t}) \pi(\rd x)} \, \pi(\rd x) , 
\end{eqnarray}
We conclude with the following observation. In the case of Brownian information, it is well 
known (and implicit in the example of Wiener 1948) that if the signal is  Gaussian, then the 
optimal filter is a linear function of the observation $\xi_t$. One might therefore ask in the 
case of a gamma information process if some special choice of the signal distribution gives 
rise to a linear filter. The answer is affirmative. Let $U$ be a gamma-distributed random
 variable with the distribution 
\begin{eqnarray}
{\mathbb P}({\mathit U} \in \rd u)= {\mathds 1}\{u>0\}\, \frac{ \theta ^{r} u^{r-1} 
\exp{(- \theta u )} } { \Gamma[r]} \, \rd u, 
\end{eqnarray}
where $r>1$ and $\theta>0$ are parameters, and set $X=1-U$. Let $\{\xi_t\}$ be a gamma 
information process carrying message $X$, let $Y=\psi_0'(X)=m/(1-X)$, and set 
$\tau = (r - 1)/m$. Then the optimal filter for $Y$ is given by
\begin{eqnarray}
{\hat Y}_t:={\mathbb E}^{\mathbb P}[Y|{\mathcal F}_t]= \frac {\xi_t+\theta } 
{t+\tau}. 
\end{eqnarray}

\vspace{0.3cm}
\noindent {\bf Example 4: Variance-gamma information}. 
The so-called variance-gamma or VG process (Madan \& Seneta 1990, Madan \& Milne 1991, 
Madan \textit{et al}.~1998) was introduced in the theory of finance. The relevant definitions and 
conventions are as follows.  By a VG process with drift $\mu \in {\mathds R}$, volatility 
$\sigma \geq 0$, and rate $m>0$, we mean a L\'evy process with exponent 
\begin{eqnarray}
\psi(\alpha) = -m \ln \left(1-\frac{\mu}{m}\, \alpha - \frac{\sigma^2}{2m}\, 
\alpha^2  \right).
\label{VG exponent}
\end{eqnarray}
The VG process admits representations in terms of simpler L\'evy processes. Let 
$\{\gamma_t\}$ be a standard gamma process on $({\mathit\Omega},{\mathcal F},
{\mathbb P})$, with rate $m$, as defined in the previous example, and let $\{B_t\}$ be a 
standard Brownian motion, independent of $\{\gamma_t\}$. We call the scaled process 
$\{{\mathit\Gamma}_t\}$ defined by ${\mathit\Gamma}_t = m^{-1} \gamma_t$ a gamma 
subordinator with rate $m$. Note that ${\mathit\Gamma}_t$ has dimensions of time and that 
${\mathbb E}^{\mathbb P}[{\mathit\Gamma}_t]=t$. A calculation shows that the L\'evy process 
$\{V_t\}$ defined by
\begin{eqnarray}
V_t = \mu {\mathit\Gamma}_t + \sigma B_{{\mathit\Gamma}_t}
\end{eqnarray} 
has the exponent  (\ref{VG exponent}). The VG process thus takes the form of a Brownian 
motion with drift, time-changed by a gamma subordinator. If $\mu = 0$ and $\sigma=1$, 
we say that $\{V_t\}$ is a ``standard'' VG process, with rate parameter $m$. If $\mu \neq 0$, 
we say that $\{V_t\}$ is a ``drifted'' VG process. One can always choose units of time such 
that $m =1$, but for applications it is better to choose conventional units of time (seconds for 
physics, years for economics), and treat $m$ as a model parameter. 
In the limit  $\sigma \rightarrow 0$ we obtain a gamma process with rate
$m$ and scale $\mu/m$. In the limit  $m \rightarrow \infty$ we obtain a 
Brownian motion with drift $\mu$ and volatility $\sigma$. 

An alternative representation of the VG process results if we let $\{\gamma^1_t\}$ 
and $\{\gamma^2_t\}$ be independent standard gamma processes on 
$({\mathit\Omega},{\mathcal F},{\mathbb P})$, with rate $m$, and set 
\begin{eqnarray}
V_t =\kappa_1 \gamma^1_t - \kappa_2 \gamma^2_t,
\end{eqnarray}
where $\kappa_1$ and $\kappa_2$ are nonnegative constants. A calculation shows that the 
exponent is of the form (\ref{VG exponent}). In particular, we have
\begin{eqnarray}
\psi(\alpha) = -m \ln \left(1-( \kappa_1 -  \kappa_2)\, 
\alpha - \kappa_1 \kappa_2\, \alpha^2  \right),
\end{eqnarray}
where
$\mu = m (\kappa_1 - \kappa_2)$ and 
$ \sigma^2 = 2m \kappa_1  \kappa_2 \,$, or equivalently
\begin{eqnarray}
\kappa_1=   \frac {1}{2m} \left( \mu + \sqrt{\mu^2 + 2m \sigma^2} \right)    
\quad {\rm and} \quad
\kappa_2=   \frac {1}{2m} \left( -\mu + \sqrt{\mu^2 + 2m \sigma^2} \right),   
\end{eqnarray} 
where $\alpha\in \{w\in{\mathds C}: {-1/\kappa_2  <\rm Re}\,w<1/\kappa_1\}$. Now let 
$\{\xi_t\}$ be a standard VG process on $({\mathit\Omega},{\mathcal F},{\mathbb P_0})$, 
with exponent $\psi_0(\alpha) = 
-m\ln(1-(2m)^{-1}\alpha^2)$ for $\alpha\in \{w\in{\mathds C}: {| \rm Re}\,w|<\sqrt{2m}\}$. 
Under the transformed measure ${\mathbb P_{\lambda}}$ defined by the change-of-measure
martingale (\ref{Esscher martingale}), one finds that  $\{\xi_t\}$ is a drifted VG process, with 
\begin{eqnarray}
\mu =   \lambda \, \left(1- \frac{1}{2m}\, \lambda^2  \right)^{-1}     
\quad {\rm and} \quad
\sigma =   \left(1- \frac{1}{2m}\, \lambda^2  \right)^{-\half}    
\end{eqnarray}
for $| \lambda | <\sqrt{2m}$. Thus in the case of the VG process an Esscher transformation 
affects both the drift and the volatility. Note that for large $m$ the effect on the volatility is 
insignificant, whereas the effect on the drift reduces to that of an ordinary Girsanov 
transformation. 

With these facts in hand, we are now in a position to construct the VG information process. 
We fix a probability space $({\mathit\Omega},{\mathcal F},{\mathbb P})$ and a number 
$m>0$.
\begin{proposition}
Let $\{{\mathit\Gamma}_t\}$ be a standard gamma subordinator with rate $m$, let $\{B_t\}$ 
be an independent Brownian motion, and let the independent random variable $X$ satisfy 
$|X| <\sqrt{2m}$ almost surely. Then the process $\{\xi_t\}$ defined by
\begin{eqnarray}
\xi_t = X \left(1- \frac{1}{2m}\, X^2  \right)^{-1} {\mathit\Gamma}_t \, + \,
\left(1- \frac{1}{2m}\, X^2  \right)^{-\frac{1}{2}} B({\mathit\Gamma}_t) 
\label{eq:x45} 
\end{eqnarray}
is a L\'evy information process with message $X$ and VG noise, with fiducial exponent 
\begin{eqnarray}
\psi_0(\alpha) = -m\ln \left(1- \frac{1}{2m}\, \alpha^2  \right)
\label{standard VG}
\end{eqnarray} 
for $\alpha\in \{w\in{\mathds C}: {\rm Re}\,w<\sqrt{2m}\}$. 
\label{prop:9}
\end{proposition}

\begin{proof}
Observe that $\{\xi_t\}$ is $\mathcal F^X $-conditionally a drifted VG process of the
form 
\begin{eqnarray}
\xi_t = \mu_X {\mathit\Gamma}_t + \sigma_X B({\mathit\Gamma}_t), 
\end{eqnarray}
where the drift and volatility coefficients are 
\begin{eqnarray}
\mu_X =   X \, \left(1- \frac{1}{2m}\, X^2  \right)^{-1}     
\quad {\rm and} \quad
\sigma_X =   \left(1- \frac{1}{2m}\, X^2  \right)^{-\frac{1}{2}}.   
\end{eqnarray}
The $\mathcal F^X $-conditional ${\mathbb P}$-exponent of $\{\xi_t\}$ is by
(\ref{VG exponent}) thus given for $\alpha \in \mathds C^{\rm I}$ by
\begin{eqnarray}
\psi_X(\alpha) 
&=& -m \ln \left(1-\frac{1}{m} \mu_X \, \alpha - \frac{1}{2m} \sigma_X^2\, \alpha^2  \right)
\nonumber \\ 
&=& -m \ln \left(1-\frac{1}{m} X  \left(1- \frac{1}{2m}\, X^2  \right)^{-1}\, \alpha - \frac{1}{2m} 
\left(1- \frac{1}{2m}\, X^2  \right)^{-1}\, 
\alpha^2  \right)
\nonumber \\ 
&=& -m\ln \left(1- \frac{1}{2m}\, (X + \alpha)^2  \right) +m\ln \left(1- \frac{1}{2m}\, X^2  \right),
\end{eqnarray}
which is evidently by  (\ref{standard VG}) of the form $ \psi_0(X + \alpha) - \psi_0(X)$, as required. 
 \qed 
\end{proof}

An alternative representation for the VG information process can be established by the same 
method if one randomly rescales the gamma subordinator appearing in the time-changed 
Brownian motion. The result is as follows. 

\begin{proposition}
Let $\{{\mathit\Gamma}_t\}$ be a gamma subordinator with rate $m$, let $\{B_t\}$ be an 
independent standard Brownian motion, and let the independent random variable $X$ satisfy 
$|X|<\sqrt{2m}$ almost surely. Write $\{ {\mathit\Gamma}^X_t \}$ for the subordinator: 
\begin{eqnarray}
{\mathit\Gamma}^X_t =  \left(1- \frac{1}{2m}\, X^2  \right)^{-1} {\mathit\Gamma}_t \, .
\end{eqnarray}
Then the process $\{\xi_t\}$ defined by
$\xi_t = X {\mathit\Gamma}^X_t  + B({\mathit\Gamma}^X_t)$
is a VG information process with message $X$.
\label{prop:10}
\end{proposition}

A further representation of the VG information process arises as a consequence of the 
representation of the VG process as the asymmetric difference between two independent 
standard gamma processes. In particular, we have: 

\begin{proposition}
Let $\{\gamma^1_t\}$ and $\{\gamma^2_t\}$ be 
independent standard gamma processes, each with rate $m$, and let the independent random 
variable $X$ satisfy
$ |X| <\sqrt{2m}$ almost surely. Then the process $\{\xi_t\}$ defined by
\begin{eqnarray}
\xi_t = \frac{1}{\sqrt{2m}-X} \, \gamma^1_t - 
\frac{1}{\sqrt{2m} +X}\, \gamma^2_t
\label{eq:x53}
\end{eqnarray}
is a VG information process with message $X$.
\label{prop:11}
\end{proposition}

\vspace{0.3cm}
\noindent {\bf Example 5: Negative-binomial information}. 
By a negative binomial process with rate parameter $m$ and probability parameter $q$, where 
$m > 0$ and $0 < q <1$, we mean a L\'evy process  with exponent 
\begin{eqnarray}
\psi_0(\alpha) = m\ln\left(\frac{1-q}{1-q\re^{\alpha}}\right)  
\label{eq:w84}
\end{eqnarray}
for $\alpha \in  \{w \in {\mathds C}\, |\, {\rm Re}\,w < -\ln q\}$. There are two representations for 
the negative binomial process (Kozubowski \& Podg\'orski 2009; Brody \textit{at al}. 2012). The 
first of these is a compound Poisson process for which the jump size $J\in{\mathds N}$ has a 
logarithmic distribution  
\begin{eqnarray}
{\mathbb P_0}(J= n) = - \frac {1}{\ln (1-q)} \,\,\frac { 1 } { n } \, q^n\, , 
\label{NB jump distribution}
\end{eqnarray}
and the intensity  of the Poisson process determining the timing of the jumps is given by 
$\lambda = -m\ln(1-q)$. 
One finds that the characteristic function of $J$ is
\begin{eqnarray}
\phi_0(\alpha) := \mathbb E^{\mathbb P_0} [\exp (\alpha J)] = 
\frac { \ln (1-q \re^\alpha)}  { \ln (1-q)} 
\end{eqnarray}
for $\alpha \in  \{w \in {\mathds C}\, |\, {\rm Re}\,w < -\ln q\}$. Then if we set 
\begin{eqnarray}
n_t = \sum_{k=1}^{\infty} {\mathds 1}\{k \leq N_t\}\,J_k , 
\end{eqnarray}
where $\{N_t\}$ is a Poisson process with rate $\lambda$, and 
$\{J_k\}_{k \in {\mathds N}}$ denotes a collection of independent identical copies of 
$J$, representing the jumps, one deduces  that 
\begin{eqnarray}
{\mathbb P_0}(n_t= k) =  \frac{\Gamma(k+mt)}{\Gamma(mt)\Gamma(k+1)} q^k (1-q)^{mt}, 
\end{eqnarray}
and that the resulting exponent is given by (\ref{eq:w84}). 
The second representation of the 
negative binomial process makes use of the method of subordination. We take a 
Poisson process with rate $\Lambda = mq/(1-q)$, and  time-change it using a gamma 
subordinator $\{{\mathit\Gamma}_t\}$ with rate parameter $m$. The moment 
generating function thus obtained, in agreement with (\ref{eq:w84}), is
\begin{eqnarray}
{\mathbb E}^{\mathbb P_0}\left[\exp\big(\alpha N({\mathit\Gamma}_t)\big) \right ] =
{\mathbb E}^{\mathbb P_0}\left[\exp \left(\Lambda (\re^{\alpha} - 1) 
{\mathit\Gamma}_t \right)\right] 
= \left( \frac{1-q}{\,1-q\re^{\alpha}}\right)^{mt} .
\end{eqnarray}
With these results in mind, we fix a probability space 
$({\mathit\Omega},{\mathcal F},{\mathbb P})$ and find:

\begin{proposition}
Let $\{{\mathit\Gamma}_t\}$ be a gamma subordinator with rate $m$, let $\{N_t\}$ be 
an independent Poisson process with rate $m$, let the independent random variable $X$ satisfy 
$X <-\ln q$ almost surely, and set 
\begin{eqnarray}
{\mathit\Gamma}^X_t = \left( \frac{q \re^X}{1-q\re^X} \right) {\mathit\Gamma}_t.
\end{eqnarray} 
Then the process $\{\xi_t\}$ defined by 
\begin{eqnarray}
\xi_t = N({\mathit\Gamma}_t^X ) 
\end{eqnarray} 
is a L\'evy information process with message $X$ and negative binomial noise, with 
fiducial exponent \rm{(\ref{eq:w84})}.
\label{prop:12}
\end{proposition}

%
%
\begin{proof}
This can be verified by direct calculation. For $\alpha \in {\mathds C^{\rm I} }$ we have: 
\begin{eqnarray}
{\mathbb E}^{\mathbb P}\left[\re^{\alpha \xi_t}|X\right] 
&=& 
 \left. {\mathbb E}^{\mathbb P}\left[\exp(\alpha 
N({\mathit\Gamma}_t^X))\right| X\right] 
= {\mathbb E}^{\mathbb P} 
\left[ \left. \exp\left(m\frac{q\re^{X}}{1-q\re^{X}}(\re^{\alpha}-1)\mathit \Gamma_t \right)
\right|X \right] 
\nonumber \\ 
&&\quad=
\left(1-\frac{q\re^{X}\left(\re^{\alpha}-1\right)}{1-q\re^{X}} \right)^{-mt}  
= 
\left(\frac{1-q\re^{X}}{1-q\re^{X+\alpha}}\right)^{mt} , 
\end{eqnarray}
which by  (\ref{eq:w84}) shows that the conditional exponent is $\psi_0(X+\alpha) - \psi_0(X) $. 
 \qed 
\end{proof}

There is also a representation for negative binomial information based on the compound 
Poisson process. This can be obtained by an application of Proposition~\ref{prop:6}, which 
shows how the L\'evy measure transforms under a random Esscher transformation. In the 
case of a negative binomial process with parameters $m$ and $q$, the L\'evy measure is 
given by
\begin{eqnarray}
\nu ( {\rd} z) = m \sum_{n=1}^{\infty} \frac{1}{n} \, q^n \, \delta_n ({\rd} z),
\end{eqnarray} 
where $\delta_n ({\rm d} z)$ denotes the Dirac measure with unit mass at the point $z=n$. 
The L\'evy measure is finite in this case, and we have $\nu (\mathds R) = -m \ln (1-q)$, 
which is the overall rate at which the compound Poisson process jumps. If one normalises 
the L\'evy measure with the overall jump rate, one obtains the probability measure 
(\ref{NB jump distribution}) for the jump size. With these facts in mind, we fix a probability 
space $({\mathit\Omega},{\mathcal F},{\mathbb P})$ and specify the constants $m$ and $q$, 
where $m>1$ and $0 < q < 1$. Then as a consequence of Proposition 6 we have the following:

\begin{proposition}
Let the random variable $X$ satisfy 
$X <-\ln q$ almost surely, let the random variable $J^X$ have the conditional distribution
\begin{eqnarray}
{\mathbb P}(J^X= n \,| \, X) = - \frac {1}{\ln (1-q\re^X)} \,\,\frac { 1 } { n } \, (q\re^X)^n\, , 
\end{eqnarray}
let $\{J^X_k\}_{k \in {\mathds N}}$ be a collection of conditionally independent identical copies 
of $J^X$, and let $\{N_t\}$ be an independent Poisson process with rate $m$. Then the process 
$\{\xi_t\}$ defined by
\begin{eqnarray}
\xi_t = \sum_{k=1}^{\infty} {\mathds 1}\{k \leq N(-\ln (1-q\re^X)t )\}\,J^X_k
\end{eqnarray}
is a L\'evy information process with message $X$ and negative binomial noise, with fiducial exponent 
\rm{(\ref{eq:w84})}.
\label{prop:13}
\end{proposition}
 
\vspace{0.3cm}
\noindent {\bf Example 6: Inverse Gaussian information}. 
The inverse Gaussian (IG) distribution appears in the study of the first exit time of Brownian 
motion with drift (Schr\"odingier 1915). The name ``inverse Gaussian'' was introduced by 
Tweedie (1945), and a L\'evy process whose increments have the IG distribution was 
introduced in Wasan (1968). By an IG process with parameters $a>0$ and $b>0$,  we mean 
a L\'evy process with exponent 
\begin{eqnarray}
\psi_0(\alpha) = a\left(b-\sqrt{b^2-2\alpha}\right)
\label{eq:x60}
\end{eqnarray}
for $\alpha \in \{w \in {\mathds C}\, |\, 0\leq {\rm Re}\,w < \half b^2\}$. Let us 
write $\{G_t\}$ for the IG process. The probability density function for $G_t$ is
\begin{eqnarray}
{\mathbb P_0}(G_t \in \rd x) =  {\mathds 1}\{x>0\}\, \frac{at}{\sqrt{2\pi x^3}}\, 
\exp\left( -\frac{(bx-at)^2}{2x} \right) \rd x,  
\end{eqnarray}
and we find that  ${\mathbb E}^{\mathbb P_0}[G_t]=at/b$ and that 
${\rm Var}^{\mathbb P_0}[G_t]=at/b^3$. It is straightforward to check that under the 
Esscher transformation ${\mathbb P}_0 \rightarrow {\mathbb P}_{\lambda}$ induced 
by (\ref{Esscher martingale}), where $0<\lambda <\half b^2$, the parameter $a$ is left 
unchanged, whereas $b \rightarrow  (b^2-2\lambda)^{1/2}$. With these facts in mind we 
are in a position to introduce the associated information process. We fix a probability 
space $({\mathit\Omega},{\mathcal F},{\mathbb P})$ and find the following:  

\begin{proposition}
Let $G(t)$ be an inverse Gaussian process with parameters $a$ and $b$, let $X$ be an 
independent random variable satisfying $0<X<\half b^2$ almost surely, and set 
$Z= b^{-1}(b^2 - 2X)^{1/2}$. Then the process $\{\xi_t\}$ defined by 
\begin{eqnarray}
\xi_t = Z^{-2} G(Z t)
\end{eqnarray}
is a L\'evy  information process with message $X$ and inverse Gaussian noise, with fiducial 
exponent {\rm (\ref{eq:x60})}. 
\label{prop:14}
\end{proposition}

\begin{proof}
It should be evident by inspection that $\{\xi_t\}$ is $\mathcal F^X$-conditionally L\'evy. Let 
us therefore work out the conditional exponent. For $\alpha \in {\mathds C^{\rm I} }$ we have:
\begin{eqnarray}
{\mathbb E}^{\mathbb P}\left[\exp(\alpha \,\xi_t)|X\right] && \nonumber \\
&& \hspace{-2.3cm} 
= \, {\mathbb E}^{\mathbb P}\left[ 
\left. \exp\left(\alpha \, \frac{ b^2}{b^2-2X}\, G\left(b^{-1}\sqrt{b^2-2X}\, t\right)\right)\right| X \right]
\nonumber \\
&& \hspace{-2.3cm} = \,\exp \left(at\left(\sqrt{b^2-2X}-\sqrt{b^2-2(\alpha +X)}\right)\right)
\nonumber \\
&& \hspace{-2.3cm} = \, 
\exp \left(at\left(b-\sqrt{b^2-2(\alpha +X)}\right) - at\left( b -\sqrt {b^2-2X}\right)\right), 
\end{eqnarray}
which shows that the conditional exponent is of the form $\psi_0(\alpha+X)-\psi_0(X)$. 
 \qed 
\end{proof}

\vspace{0.3cm}
\noindent {\bf Example 7: Normal inverse Gaussian information}.
By a normal inverse Gaussian (NIG) process (Rydberg 1997, Barndorff-Nielsen 1998) 
with parameters $a$, $b$, and $m$, such that $a>0$, $|b| < a$, and $m>0$, we mean a 
L\'evy process with an exponent of the form
\begin{eqnarray}
\psi_0(\alpha) = m \left(\sqrt{a^2 -b^2}-\sqrt{a^2-(b+\alpha)^2}\right)  
\label{eq:z69} 
\end{eqnarray}
for $\alpha\in \{w\in{\mathds C}: { -a-b  <\rm Re}\,w<a-b\}$. Let us 
write $\{I_t\}$ for the NIG process. The probability density for its value 
at time $t$ is given by
\begin{eqnarray}
{\mathbb P_0}(I_t \in \rd x) = \frac{am t K_1\big( a\sqrt{m^2t^2+x^2} \big)}
{\pi\sqrt{m^2t^2+x^2}}\, 
\exp\left( m t\sqrt{a^2-b^2}+bx \right) \rd x,  
\end{eqnarray}
where $K_\nu$ is the modified Bessel function of third kind (Erd\'elyi 1953). 
The 
NIG process can be 
represented as a Brownian motion subordinated by an IG process. In particular, let 
$\{B_t\}$ be a standard Brownian motion, let $\{G_t\}$ be an independent IG  process with 
parameters $a'$ and $b'$, and set $a' =1$ and $b' = m (a^2-b^2)^{1/2}$. Then the 
characteristic function of the process $\{I_t\}$ defined by 
\begin{eqnarray}
I_t = b m^2G_t + m B(G_t)
\end{eqnarray}
 is 
given by (\ref{eq:z69}). The associated information process is constructed as follows. We 
fix a probability space $({\mathit\Omega},{\mathcal F},{\mathbb P})$ and the parameters 
$a$, $b$, and $m$. 

\begin{proposition}
Let the random variable $X$ satisfy $-a-b  < X <a-b$ almost surely, let $\{G_t^X\}$ be 
$\mathcal F^X$-conditionally IG, with parameters $a' = 1$ and $b' =m (a^2-(b+X)^2)^{1/2}$, 
and let $F_t = m^2 G_t^X$. Then the process $\{\xi_t\}$ defined by 
\begin{eqnarray}
\xi_t = (b+X) F_t + B(F_t)
\end{eqnarray}
is a L\'evy  information process with message $X$ and NIG noise, with fiducial exponent 
{\rm (\ref{eq:z69})}. 
\label{prop:15}
\end{proposition}

\begin{proof}
We observe that the condition on $\{G_t^X\}$ is that
\begin{align}
\frac{1}{t} \ln {\mathbb E}^{\mathbb P}\!\left[ \exp\left(\alpha G_t^X\right) |X \right] 
= \delta \sqrt{a^2-(b+X)^2} \!-\! \sqrt{m^2(a^2-(b+X)^2)-2\alpha} 
\end{align}
for $\alpha\in{\mathds C^{\rm I}}$. 
Thus if we set  $\psi_X(\alpha) = {\mathbb E}^{\mathbb P}\left[ 
\exp (\alpha \xi_t) |X\right]$ for $\alpha\in{\mathds C}^{\rm I}$ it follows that
\begin{eqnarray} 
\hspace{0.5cm} 
\psi_X(\alpha) &=& {\mathbb E}^{\mathbb P} 
\left[\left.\exp \left( \alpha(b+X)F_t + \alpha B(F_t) \right)\right|X\right] 
\nonumber \\  &=&  
{\mathbb E}^{\mathbb P} \left[\left. 
\exp \left( (\alpha(b+X) + \half \alpha^2)m^2 G_t^X \right) \right| X \right]  \\ 
&& \hspace{-2.3cm} = 
{\mathbb E}^{\mathbb P}\! \left[ \exp \left( m t\sqrt{a^2-(b+X)^2} - 
m t \sqrt{a^2-(b+X)^2 - 2\left(\alpha(b+X)+\half \alpha^2\right)} \right) 
\right], \nonumber
\end{eqnarray}
which shows that the conditional exponent is of the required form. 
 \qed 
\end{proof}

\vspace{0.1cm}

Similar  arguments  lead to the construction of information processes based on various 
other L\'evy processes related to the IG distribution, including for example the 
generalised hyperbolic process (Barndorff-Nielsen 1977), for which the information 
process can be shown to take the form 
\begin{eqnarray}
\xi_t = (b+X) {\tilde G}_t + B({\tilde G}_t). 
\end{eqnarray}
Here the random variable $X$ is taken to be ${\mathbb P}$-independent of the standard 
Brownian motion $\{B(t)\}$, and $\{{\tilde G}_t\}$ is $\mathcal F^X$-conditionally a generalised 
IG process with parameters $(\delta, (a^2-(b+X)^2)^{1/2} ,\nu)$. It would be of interest to 
determine whether explicit models can be obtained for information processes based on the 
Meixner process  (Schoutens \& Teugels 1998) and the CGMY process 
(Carr \textit{et al}.~2002, Madan \& Yor 2008). 

We conclude this study of L\'evy information with the following remarks. Recent developments 
in the phenomenological representation of physical (Brody \& Hughston 2006) and economic 
(Brody \textit{et al}.~2008a) time series have highlighted the idea that signal processing 
techniques may have far-reaching applications to the identification, characterisation and 
categorisation of phenomena, both in the natural and in the social sciences, and that beyond 
the conventional remits of \textit{prediction}, \textit{filtering}, and \textit{smoothing} there 
is a fourth and important new domain of applicability: the \textit{description} of phenomena in 
science and in society. It is our hope therefore that the theory of signal processing with 
L\'evy information herein outlined will find a variety of interesting and exciting applications. 

\vspace{0.3cm}

\noindent {\footnotesize 
The research reported in this paper has been supported in part by Shell Treasury 
Centre Limited, London, and by the Fields Institute, University of Toronto.
The authors are grateful to N. Bingham, 
M. Davis, E. Hoyle, M. Grasselli, T. Hurd, S. Jaimungal,  E. Mackie, A. Macrina, P. Parbhoo and 
M. Pistorius for helpful comments and discussions.}


\end{document}